\documentclass[reqno,12pt]{amsart}
\usepackage[utf8]{inputenc}
\usepackage[left=1.2in, right=1.2in, top=1in]{geometry}

\usepackage{makecell}
\usepackage{esint}
\usepackage{amsmath,amssymb,amsthm,mathrsfs,color,times,textcomp,verbatim,mathtools}
\usepackage{graphicx}
\usepackage{xcolor}
\usepackage[colorlinks=true]{hyperref}
\hypersetup{urlcolor=blue, citecolor=red, linkcolor=blue}
\usepackage[square,numbers]{natbib}
\usepackage[skip=5pt, indent=13pt]{parskip}

\numberwithin{equation}{section}
\theoremstyle{plain} 
\newtheorem{theorem}{Theorem}[section]
\newtheorem*{theorem*}{Theorem}
\newtheorem{proposition}[theorem]{Proposition}
\newtheorem{lemma}[theorem]{Lemma}

\newtheorem{corollary}[theorem]{Corollary}
\newtheorem{claim}{Claim}
\newtheorem{remark}[theorem]{Remark}
\newtheorem{definition}[theorem]{Definition}

\newtheorem{innercustomthm}{Theorem}
\newenvironment{customthm}[1]{\renewcommand\theinnercustomthm{#1}\innercustomthm}{\endinnercustomthm}

\def\Rnp{\mathbb{R}^N_+}
\def\Rn{\mathbb{R}^N}
\def\Bn{\mathbb{B}^N}
\def\r{\mathfrak{r}}

\title[Hardy-Sobolev inequalities and their best constants]{Attainability of the best constant of Hardy-Sobolev inequality with full boundary singularities}

\author{Liming Sun}
\address{Academy of Mathematics and Systems Sciences, Chinese Academy of Sciences, Beijing, 100190, China.}
\email{lmsun@amss.ac.cn}

\author{Lei Wang}
\address{School of Mathematics, Shandong University, Jinan, 250100, China.}
\email{wangleiwl@sdu.edu.cn}

 \date{\today \,(Last Typeset)}
 \subjclass[2020]{35A23 (Primary), 35B09, 35B44, 35J75, 35J91 (Secondary).}
 \keywords{Hardy-Sobolev inequality, distance function, best constant, least-energy solution.}

\begin{document}
\begin{abstract}
   We consider a type of Hardy-Sobolev inequality, whose weight function is singular on the whole domain boundary. We are concerned with the attainability of the best constant of such inequality. In dimension two, we link the inequality to a conformally invariant one using the conformal radius of the domain. The best constant of such inequality on a smooth bounded domain is achieved if and only if the domain is non-convex. In higher dimensions, the best constant is achieved if the domain has negative mean curvature somewhere. If the mean curvature vanishes but is non-umbilic somewhere, we also establish the attainability for some special cases. In the other direction, we also show the best constant is not achieved if the domain is sufficiently close to a ball in  $C^2$ sense.
\end{abstract}
\maketitle

\section{Introduction}
\subsection{Statement of main results}When $N\geq 3$, the well-known Hardy's inequality states that 
\begin{align}\label{Hardy-|x|}
    \left(\frac{N-2}{2}\right)^2\int_{\Rn}\frac{|u(x)|^2}{|x|^2}dx\leq \int_{\Rn}|\nabla u|^2dx,\quad \forall \,u\in C_0^\infty(\Rn).
\end{align}
The constant $(N-2)^2/4$ is optimal and can not be attained by any non-trivial function in $\mathcal{D}^{1,2}(\Rn)$, the completion of $C_0^\infty(\Rn)$ under the norm $(\int_{\Rn}|\nabla u|^2dx)^{\frac{1}{2}}$. For any Lipschitz domain $\Omega\subset \Rn$ with $N\geq 3$, we have the Sobolev inequality
\begin{align}\label{Sobolev}
\mu^S(\Omega)\left(\int_{\Omega}u^{\frac{2N}{N-2}}\right)^{\frac{N-2}{N}}\leq \int_{\Omega}|\nabla u|^2,\quad \forall\, u\in C^\infty_0(\Omega).
\end{align}
It is well-known that the best constant $\mu^S(\Omega)=\mu^S(\Rn)$ for any domain $\Omega$ and $\mu^S(\Omega)$ is never attained unless $\Omega=\Rn$.



H\"older-interpolating Hardy's inequality and Sobolev inequality, one can obtain the so-called \textit{Hardy-Sobolev} type inequality. For instance, using \eqref{Hardy-|x|} and \eqref{Sobolev}, one can establish the following inequality for any Lischitz domain $\Omega\subset \Rn$ with $N\geq 3$.
\begin{align}\label{HS-|x|}
     C\left(\int_{\Omega}\frac{|u|^{2^*(s)}}{|x|^s}\right)^{2/2^*(s)}\leq \int_{\Omega }|\nabla u|^2,\quad \forall\, u\in C^\infty_0(\Omega)
\end{align}
where $2^*(s)=\frac{2(N-s)}{N-2}$ and $s\in (0,2)$. When $\Omega=\Rn$, the above one is a particular case of Caffarelli–Kohn–Nirenberg inequality (cf. \cite{kohn1984first}), which has received considerable attention.

When $0$ is in the interior of $\Omega$, it is easy to show the best constant in \eqref{HS-|x|} is independent of $\Omega$ and not achieved unless $\Omega=\Rn$, see, for instance, \citet{ghoussoub2000multiple}, \citet{lieb1983sharp}, \citet{catrina2001caffarelli}. New phenomena happen when $0$ is on the boundary of $\Omega$. \citet{egnell1992positive} shows that the best constant for any cone with vertex at $0$ can be achieved. If the mean curvature of $\partial \Omega$ at the origin is negative then \citet{ghoussoub2006effect} show that the best constant can also be achieved (see also \citet{chern2010minimizers}). On the other hand, if $\Omega$ is star-shaped around $0$, then \citet{ghoussoub2004hardy} show that it is not attained. There are many interesting works related to \eqref{HS-|x|} which are impossible to enumerate. One can consult the survey \cite{ghoussoub2016sobolev} and references therein.

For domains with boundaries, there is also a version of Hardy's inequality formulated in terms of the distance function from points in the domain to the boundary. To be more precise, let $\Omega$ be a domain in $\Rn$ ($N\geq 2$) with non-empty boundary. Consider the following variational problem
\begin{align}\label{def:muH}
    \mu^H(\Omega)=\inf\left\{\int_{\Omega}|\nabla u|^2: u\in C^\infty_0(\Omega) \text{ and } \int_{\Omega}\left(\frac{u}{\delta_\Omega}\right)^2=1\right\}
\end{align}
where $\delta_\Omega(x)=\text{dist}(x,\partial \Omega)$. The infimum being positive is equivalent to the validity of the following Hardy's inequality
\begin{align}\label{Hardy}
\mu^H(\Omega)\int_{\Omega}\left(\frac{u}{\delta_\Omega}\right)^2\leq \int_{\Omega}|\nabla u|^2,\quad \forall\,u\in C^\infty_0(\Omega).
\end{align}

The infimum of \eqref{def:muH} is the so-called \textit{Hardy constant} of $\Omega$. It is known that $\mu^H(\Omega)>0$ for a large class of domains, including bounded Lipschitz ones, for instance, see \citet{OK1990Hardy}. If $\partial \Omega$ has a tangent plane at least at one point, then $\mu^H(\Omega)\leq 1/4$, see \citet{marcus1998best}. If $\Omega$ is convex, then \cite{marcus1998best}  and Matskewich et al. \cite{MS1997best} also show that $\mu^H(\Omega)=1/4$. This is established for weakly mean convex $C^2$ domains independently by \cite{psaradakis20131,lewis2012geometric,giga2013representation}. For the attainability of $\mu^H(\Omega)$, it was proved in \cite{marcus1998best} that, for bounded $C^2$ domains, the infimum in \eqref{def:muH} is achieved if and only if $\mu^H(\Omega)<1/4$. There are works on Hardy's inequality and its improvement which are impossible to list all, for instance, see \cite{brezis1997hardy,filippas2007critical,tertikas2007existence,benguria2007sharp,ghoussoub2011bessel,lam2020geometric,LP2019,DPD23On} and references therein. 

 H\"older-interpolating \eqref{Hardy} and \eqref{Sobolev} gives the following type of Hardy-Sobolev inequality. 
 For Lipschitz domains $\Omega\subset\Rnp$ with $N\geq 3$, one has 
\begin{align}\label{HS-3}
    C\left(\int_{\Omega}\frac{|u|^{2^*(s)}}{\delta_\Omega^s}\right)^{2/2^*(s)}\leq \int_{\Omega }|\nabla u|^2,\quad \forall\, u\in C^\infty_0(\Omega).
\end{align}
 There appears to be limited research conducted on it.
 It seems the above one was first studied by \citet{chen2007hardy}. The second-named author and Zhu also consider the above inequality in \cite{wang2022hardy}. Moreover, they also study an inequality for $N=2$ on a Lipschitz domain $\Omega\subset \mathbb{R}^2$,
\begin{align}\label{HS-2}
    C\left(\int_{\Omega}\frac{|u|^{p}}{\delta_\Omega^2}\right)^{2/p}\leq \int_{\Omega }|\nabla u|^2,\quad \forall\, u\in C^\infty_0(\Omega)
\end{align}
where $p\in (2,\infty)$. Although inequality \eqref{HS-2} does not arise directly from H\"older-interpolating Hardy and Sobolev ones, it is reminiscent of \eqref{HS-3} in $N=2$. We shall call both of them Hardy-Sobolev inequalities with full boundary singularities.

The best constants in \eqref{HS-3} and \eqref{HS-2} are\footnote{In $N\geq 3$, it seems that $\mu_s(\Omega)$ is more appropriate than $\mu_p(\Omega)$. This is the notation in \cite{chen2007hardy,ghoussoub2004hardy}. However, to include the case $N=2$, we have to use $\mu_p(\Omega)$.}
\begin{align}\label{def:muHS}
    \mu_p(\Omega)=\inf\left\{\int_{\Omega}|\nabla u|^2: u\in C^\infty_0(\Omega) \text{ and } \int_{\Omega}\frac{|u|^{p}}{\delta_\Omega^s}=1\right\}.
\end{align}
Here we make the following assumption on $p$ and $s$ throughout this paper (unless specified),
\begin{align}\label{assum-pnu}
\begin{cases}
s\in (0,2), \quad p=2^*(s)=\frac{2(N-s)}{N-2}, & \text{if }N\geq 3,\\
s=2,\quad\quad \ p\in (2,\infty),&\text{if }N=2.
\end{cases}
\end{align}
The positivity of the infimum is equivalent to the validity of \eqref{HS-3} and \eqref{HS-2} with $C=\mu_p(\Omega)$. Note that when $p=\frac{2N}{N-2}$ for $N\geq 3$ (that is $s=0$), \eqref{HS-3} reduces to Sobolev inequality \eqref{Sobolev} and $\mu_{\frac{2N}{N-2}}(\Omega)=\mu^S(\Omega)$. When $p=2$, \eqref{HS-3} and \eqref{HS-2} reduce to Hardy's inequality \eqref{Hardy} and $\mu_2(\Omega)=\mu^H(\Omega)$. Moreover, for $\Omega=\mathbb{R}^N_+$, both $\mu^S(\mathbb{R}^N_+)$ and $\mu^H(\mathbb{R}^N_+)$ are not attainable. 

As far as we know, for $p$ and $s$ satisfy \eqref{assum-pnu}, there are not many studies on this type of Hardy-Sobolev inequality unless the works \cite{chen2007hardy,wang2022hardy} mentioned above. We summarize the results of \cite{wang2022hardy} in the following (some of them are already obtained in \cite{chen2007hardy} for $N\geq 3$).

\begin{customthm}{A}\label{thm:WangZhu2}
Assume that $N\geq 2$, $p$ and $s$ satisfy \eqref{assum-pnu}. 
\begin{enumerate}
    \item $\mu_p(\Rnp)>0$ and it is attainable.
    \item If $\partial \Omega$ possesses a tangent plane at least at one point, then $\mu_p(\Omega)\leq \mu_p(\Rnp)$. If $\Omega$ is a bounded domain with Lipschitz boundary, then $\mu_p(\Omega)>0$.
    \item Let $\Omega$ be a bounded domain with $C^1$ boundary.  If $\mu_p(\Omega)<\mu_p(\Rnp)$, then $\mu_p(\Omega)$ is achieved by some $u\in H^1_0(\Omega)$.
    \item For any ball $B\subset \Rn$, $\mu_p(B)=\mu_p(\Rnp)$ and it is not attained by any $u\in H^1_0(B)$.
\end{enumerate}
\end{customthm}
Note that part (1) already shows that $\mu_p(\Rnp)$ has distinct behavior when $p$ are in the intermediate values and endpoints of $[2,\frac{2N}{N-2}]$. In part (1), the attainability of $\mu_p(\Rnp)$ and its extremals actually use authors' previous work \cite{dou2022divergent}. The extremal can be classified, but one has explicit expression only when $p=1+2/N$ and $p=1+4/N$. See Theorem \ref{thm:WangZhu} in Section 2 for a precise statement. Part (3) is the standard Brezis-Nirenberg argument for the compactness of minimizing sequences with energy below a certain threshold.  Part (4) uses the conformal equivalence of balls and $\Rnp$.

This paper concerns the following two questions:
\begin{enumerate}
     \item[(Q1)] What kind of domain do we have $\mu_p(\Omega)<\mu_p(\Rnp)$? 
    \item[(Q2)] For a general smooth convex domain, is $\mu_p(\Omega)=\mu_p(\Rnp)$? If so, is $\mu_p(\Omega)$ attainable?
\end{enumerate}

 There are some serious difficulties in attacking these problems. For instance, the loss of compactness, no moving plane available, the singularity of weight function on the whole boundary, no explicit expressions of minimizers for some cases, and so on.

 We have some partial answers for both questions. We find that the answer to (Q1) is strongly linked to the curvature of the boundary $\partial \Omega$. 

\begin{theorem}\label{thm:H<0} Suppose that $\Omega\subset \Rn(N\geq 2)$ is a domain. If $\partial \Omega $ is $C^2$ near one point and the mean curvature of this point is negative, then $\mu_p(\Omega)<\mu_p(\Rnp)$ for $p$ satisfying \eqref{assum-pnu}. Consequently, if $\Omega$ is also bounded with $C^1$ boundary then $\mu_p(\Omega)$ is attainable.
\end{theorem}

\begin{remark}\label{rmk:n>2}
In contrast, for $N\geq 3$, Theorem \ref{thm:H<0} does not hold for the endpoint case $p=2$, i.e. the Hardy case. For example,  $\mu_2(\Omega)=\frac{1}{4}$ holds for $\Omega$ being a small enough tubular neighborhood of a surface \cite{filippas2004sharp} and some convex domains with punctured balls \cite{avkhadiev2010hardy}. For $N=2$, however, we neither know the above theorem holds for the endpoint case $p=2$ nor can find $C^2$ non-convex planar domain such that its Hardy constant is $1/4$. It is worth pointing out that Davies \cite{davies1995hardy} showed that $\mu_2(\Omega)=\frac14$ for some non-convex plane sectors $\Omega$, but here $\Omega$ is not a $C^2$ domain.
\end{remark}


To prove $\mu_p(\Omega)<\mu_p(\Rnp)$, we construct some auxiliary functions under the Fermi coordinates. The computation here is inspired greatly by the work of \citet{escobar1992conformal} on the boundary Yamabe problem. However, unlike the boundary Yamabe problem, the extremal of $\mu_p(\Rnp)$ is not explicitly known except in two cases. We find some interesting Pohozaev type identity to overcome this difficulty (cf. Lemma \ref{lem:Poh}).

Next, we consider the weakly convex domain. We know from \cite{psaradakis20131,lewis2012geometric,giga2013representation} that for $p=2$, it holds that  $\mu_2(\Omega)=\mu_2(\mathbb{R}^N_+)$. But the situation may be different for $p$ satisfying \eqref{assum-pnu}. Specifically, when $p=2+2/N$ the minimizer of $\mu_p(\Rnp)$ has explicit form. We take advantage of this fact to establish the following theorem. 

\begin{theorem}\label{thm:H=0}
Suppose that $p=2+2/N$ and  $\Omega\subset \Rn(N\geq 18)$ is a domain. If $\partial \Omega$ is $C^3$ near one point, and the mean curvature vanishes at this point, but this point is non-umbilic, then $\mu_p(\Omega)<\mu_p(\Rnp)$. Consequently, if $\Omega$ is also bounded with $C^1$ boundary then $\mu_p(\Omega)$ is attainable.
\end{theorem}

To investigate (Q2), we propose to study the following
\begin{align}
    \sigma_p(\Omega)=\inf\left\{\int_{\Omega}|\nabla u|^2:u\in C^\infty_0(\Omega)\text{ and } \int_{\Omega}\frac{|u|^{p}}{\r_\Omega^s}=1\right\},
\end{align}
where $\r_\Omega$ is the harmonic radius of a Lipschitz domain $\Omega$ with non-empty boundary. In dimension 2, it also coincides with the conformal radius. See Section 2 for its definition and property. For convex domains, one has the fact that $\r_\Omega\leq 2\delta_\Omega$ and the inequality is strictly somewhere in $\Omega$ unless $\Omega=\Rnp$. Moreover, it is easy to see that $\sigma_p(\Rnp)=2^{\frac{2s}{p}}\mu_p(\Rnp)$. For $\sigma_p(\Omega)$, it has the following remarkable feature.
\begin{proposition}\label{prop:main-cf}
Assume $s\in [0,2]$, $p=2^*(s)$ when $N\geq 3$ and  $s=2$, $p\geq 2$ when $N=2$. 
\begin{enumerate}
    \item If two Lipschitz domains $\Omega$ and $\widetilde \Omega$ are conformally equivalent, then $\sigma_p(\Omega)=\sigma_p(\widetilde\Omega)$.
    \item If $p$ satisfies \eqref{assum-pnu}, $\Omega$ is convex with non-empty boundary and  $\sigma_p(\Omega)= \sigma_p(\Rnp)$, then $\mu_p(\Omega)$ is attained if and only if $\Omega$ is a half-space.
\end{enumerate}
\end{proposition}
We now consider $N=2$. By Riemann mapping theorem, any simply connected hyperbolic region (a planar domain whose complement contains at least two points) is conformal to $\mathbb{R}^2_+$. Then we have 




\begin{theorem}\label{thm:planar} For planar domains, we assume $s=2$ and $p\in [2,\infty)$. Then  

\begin{enumerate}
    \item For any simply connected Lipschitz hyperbolic region $\Omega$, one has  $\sigma_p(\Omega)=\sigma_p(\mathbb{R}^2_+)$.  In particular, $\sigma_2(\mathbb{R}_+^2)=1$.
    \item If $\Omega\subset \mathbb{R}^2$ is a convex domain with non-empty boundary, then $\mu_p(\Omega)=\mu_p(\mathbb{R}^2_+)$. Moreover, if $p=2$, then $\mu_2(\Omega)$ is never attained. If $p\in (2,\infty)$, then $\mu_p(\Omega)$ is attainable if and only if  $\Omega$ is a half-plane.
\end{enumerate}
\end{theorem}



 The application of conformal invariance to Hardy inequalities on $\mathbb{R}^2$ has its origins in the work of \citet{ancona1986strong}. Ancona utilized both the Riemann mapping theorem and Koebe’s one-quarter theorem to establish a lower bound for the Hardy constant in the context of simply connected domains. The introduction of harmonic radius here can give several novel results about Hardy inequality on convex domains. The second part above expands upon existing results in the literature in the context of Hardy's case, i.e.\,$p=2$. Notably, we refrain from imposing any smoothness assumptions on $\partial \Omega$, while \cite{marcus1998best} specifically concludes the non-attainability of the Hardy constant for $C^2$ smooth convex bounded domains. Moreover, 
a by-product of our approach is  
    \begin{align}\label{conf-R2}
     \int_{\Omega} \frac{u^2}{\r_\Omega^2}\leq \int_{\Omega}|\nabla u|^2.
    \end{align}
Note that if $\Omega$ is a convex planner domain, \eqref{conf-R2} is stronger than Hardy's inequality \eqref{Hardy}. In fact, by Lemma \ref{lem:r-property}, one has $\r_\Omega\leq 2\delta_\Omega$. Thus
\begin{align}
    \int_{\Omega}|\nabla u|^2\geq \int_{\Omega}\frac{u^2}{\r_\Omega^2}\geq \int_{\Omega}\frac{u^2}{4\delta_\Omega^2}.
\end{align}

For bounded planner domains, it is well-known that mean convexity is equivalent to convexity, see \cite[Thm 6.20]{abbena2017modern}. Combining Theorem \ref{thm:H<0} and \ref{thm:planar}, we have a rather complete answer of (Q1) and (Q2) for planar domains.
\begin{corollary}
Suppose $s=2, p\in (2,+\infty)$, and $\Omega\subset \mathbb{R}^2$ is a bounded $C^2$ domain. Then $\mu_p(\Omega)$ is attainable if and only if $\Omega$ is non-convex. 
\end{corollary}


In higher dimensions $N\geq 3$, the Liouville theorem asserts that the conformal transformations consist of translation, scaling, and inversion. The scarcity of conformal equivalence in higher dimensions makes the application of Proposition \ref{prop:main-cf} limited. Recall that Theorem \ref{thm:WangZhu2} proves the non-attainability of $\mu_p(B)$ for balls $B$. We can address (Q2) for domains that are $C^2$ close to balls in high dimensions $N\geq 3$.
\begin{theorem}\label{thm:c-ball}
    Assume $N\geq 3$, $s\in (0,2)$ and $p=2^*(s)$. For any domain $\Omega$ which is sufficiently close to a ball in $C^2$ sense, it holds  $\mu_p(\Omega)=\mu_p(\mathbb{R}^N_+)$ and $\mu_{p}(\Omega)$ is not attainable.
\end{theorem} 
The proof is inspired by the work of \citet{PAN1998791} on the Lin-Ni problem. Our problem resembles the Lin-Ni one because the minimizing sequence of $\mu_p(\Omega)$ can only blow up at the boundary. We can draw from Theorem \ref{thm:c-ball} that there is some ellipsoid $E$, such that $\mu_{p}(E)=\mu_p(\mathbb{R}^N_+)$ and $\mu_{p}(E)$ is not attainable. But it is still open whether for any ellipsoid, the best constant is equal to $\mu_p(\mathbb{R}^N_+)$ and not attainable.

\subsection{Discussion}
We shall summarize the results obtained and state some open problems. 
For smooth bounded domains $\Omega$, the attainability results we obtained are listed in table \ref{tab:cmp} to make some comparisons to that of the Hardy case. There are some cases not completely solved. For instance, we conjecture that for (smooth) bounded convex domains $\mu_p(\Omega)$ is not attainable.


\begin{table}[ht]
\begin{tabular}{|c|c|c|c|}
\hline
 Domains $\Omega$ & $p=2$ & $p=2^*(s)$, $s\in (0,2)$, $N\geq 3$&  $p>2$, $N=2$\\ \hline
  $\inf H<0$ & may or may not & A & A\\ \hline
 $H\geq 0$& N & A under further conditions & N\\ \hline
 convex & N &  {N near balls}& N \\ \hline
\end{tabular}
\vspace{0.4cm}
\caption{Attainability of $\mu_p(\Omega)$. Here A denotes attainable and N denotes not attainable. $H$ is the mean curvature of $\partial \Omega$.}
\label{tab:cmp}
\end{table}





In the level of Euler-Lagrange equation, $\mu_p(\Omega)$ is also different from the two endpoints case. The minimizers of  $\mu_p(\Omega)$, up to the multiplication of some constant, satisfy the following elliptic equation, 
\begin{align}\label{E-L}
  \begin{cases}
    \Delta u+\delta_{\Omega}^{-s}{u^{p-1}}=0,\, u>0, & \text{ in } \Omega, \\
     u=0, &\text{ on }\partial \Omega.
  \end{cases}
\end{align}
For $p$ and $s$ satisfying \eqref{assum-pnu}, the existence of solutions to the above singular elliptic equations seems to be open. For domains such that $\mu_p(\Omega)$ is not attainable, it just says there is no ground state (or least energy) solution to \eqref{E-L}, but it might still have higher energy solutions. In fact, if $\Omega$ is the unit ball in $\Rn$, \cite{senba1990dirichlet} and \cite{cheng2022estimate} have found a radial symmetric solution for \eqref{E-L}. This solution can not be the ground state one considering Theorem \ref{thm:WangZhu2}. This shows the distinctive behavior of \eqref{E-L} between $s\in (0,2)$ and $s=0$. Indeed, when $s=0$, the Pohozaev identity implies there is no solution to $\Delta u+u^{(N+2)/(N-2)}=0$ if the domain is star-shaped. 
However, in our case, the Pohozaev identity for \eqref{E-L} is too complicated to obtain some non-existence results. 



\subsection{Structure of this paper} In Section 2, we shall give some preliminary about some properties of the minimizer of $\mu_p(\Rnp)$. Harmonic radius will also be introduced. In Section 3, we will give the proof of Proposition \ref{prop:main-cf} and Theorem \ref{thm:planar}. Sections 4, 5, and 6 are devoted to proving Theorem \ref{thm:H<0}, Theorem \ref{thm:H=0}, and \ref{thm:c-ball} respectively. Finally, in the appendix, we prove  Lemma \ref{lem:U} and the non-degeneracy of the minimizers.

\section{Preliminary}

Denote $\Rnp=\{x=(x',x_N): x'\in \mathbb{R}^{N-1}, x_N>0\}$ and  $\mathcal{D}_0^{1,2}(\Rnp)$ as the completion of $C^\infty_0(\Rnp)$ under the norm $(\int_{\Rnp}|\nabla u|^2dx)^{1/2}$. Recall M\"obius transformation
\begin{align}
\begin{split}\label{Mo}
\mathcal{M}:\Rnp&\to \Bn\\
(x',x_{N})&\mapsto y=\left(\frac{2x'}{|x'|^2+(1+x_N)^2},\frac{1-|x|^2}{|x'|^2+(1+x_N)^2}\right).
\end{split}
\end{align}

In \cite{wang2022hardy}, the second named author and Meijun Zhu proved the following result.
\begin{customthm}{B}\label{thm:WangZhu}
Assume $p$ and $s$ satisfy \eqref{assum-pnu}. Then $\mu_p(\Rnp)$ 
can be attained by some $0<U\in \mathcal{D}^{1,2}_0(\Rnp)$. All minimizers of $\mu_p(\Rnp)$ consist of 
\begin{align}\label{extremals}
\{CU\left({\varepsilon}(x'-x_0),{\varepsilon}x_N\right):\forall\, C>0,\forall\,\varepsilon>0,\forall\, x_0\in\mathbb{R}^{N-1}\}.
\end{align}
Moreover, $U$ is unique in the sense that if we define $\tilde U$ on $\Bn$ by $\tilde U\circ \mathcal{M}=U|\det(D\mathcal{M})|^{\frac{2-N}{2N}}$, where $\mathcal{M}$ is the M\"obius transformation \eqref{Mo}, then $\tilde U$ is the unique radial solution of  
\begin{align}\label{uinB}
\begin{cases}
    \Delta u+\frac{2^su^{p-1}}{(1-|y|^2)^s}=0\quad &\text{in}\quad \Bn,\quad 0<u\in C^2(\Bn)\cap C^0(\overline{\Bn}),\\
    u=0 \quad &\text{on}\quad\partial \Bn.
\end{cases}
\end{align}
Furthermore, all the positive solutions of \eqref{E-L} with $\Omega=\Rnp$ in $\mathcal{D}^{1,2}_0(\Rnp)$ is of the form in \eqref{extremals}.
\end{customthm}

\begin{remark}
The equality case follows from \cite{dou2022divergent}.
In the following two cases, the extremal functions can be explicitly written out and the best constant can be calculated.
\begin{enumerate}
\item $p=2+\frac{2}{N}$ (that is $s=1+\frac{2}{N}$), one has
\begin{align}
U(x',x_N)&=(2N)^{\frac{N}{2}}\frac{x_N}{[(1+x_N)^2+|x'|^2]^{N/2}},\,\, x'\in \mathbb{R}^{N-1}, x_N>0,\label{U:2/n}\\
\tilde U(y)&=\frac{1}{2}N^{\frac{N}{2}}(1-|y|^2),\,\,\, y\in \Bn.\notag
\end{align}
\item $p=2+\frac{4}{N}$ (that is $s=\frac{4}{N}$), one can take
\begin{align}
U(x',x_N)&= [N(N+2)]^{\frac{N}{4}}\frac{x_N}{(1+|x|^2)^{N/2}},\,\, x'\in \mathbb{R}^{N-1}, x_N>0,\label{U:4/n}\\
\tilde U(y)&= \frac12[N(N+2)]^{\frac{N}{4}}(1-|y|^2)(1+|y|^2)^{-\frac{N}{2}},\,\,\, y\in \Bn.\notag
\end{align}
\end{enumerate}
\end{remark}

For other cases, although we do not know the explicit expression of solutions, we can still have the following estimates. The proof is deferred to the appendix.
\begin{lemma}\label{lem:U}
    Suppose $U$ is the unique solution defined in Theorem \ref{thm:WangZhu}. Then 
    \begin{enumerate}
        \item $U$ is cylindrical, i.e. $U(x',x_N)=U(|x'|,x_N).$
        \item There exists a positive constant $C=C(p,N)$ such that for all $x\in \Rnp$,
   \begin{align}
    \begin{split}\label{u-bound}
 C^{-1}x_N(|x|+1)^{-N}&\leq |U(x)| \leq Cx_N(|x|+1)^{-N},\\
 |\nabla U(x)| &\leq C(|x|+1)^{-N}. 
 \end{split}
\end{align}
\item If $N=2$, then $U$ attains the global maximum only at $(0',1)\in \mathbb{R}^2_+$. 
If $N\geq3$, then $U$ attains the global maximum at  finitely many points in $\Rnp$, and all of them lie on the interval $\{(0',x_N):0<x_N\leq 1\}$.
\end{enumerate}
\end{lemma}
\begin{remark}
    For $N\geq 3$ and any $s\in (0,2)$, we conjecture that $U$ has only one critical point where it takes the global maximum. One can verify this for $s=1+2/N$ and $s=4/N$ (see \eqref{U:2/n} and \eqref{U:4/n}). However, for general $s\in (0,2)$, we do not know how to prove this.
\end{remark}

Let us recall some basic facts about harmonic radius (see \cite{bandle1996harmonic}). Suppose that $\Omega$ is an Euclidean domain of dimension $N\geq2$ with Lipschitz boundary. The Green's function $G(x,y)$ of $\Omega$ with boundary Dirichlet boundary condition satisfies 
\begin{align}
\begin{cases}
\Delta_y G(x, y)  =-\delta_y(x), &\text { in } \Omega, \\
G(x,y)  =0, &\text { on } \partial \Omega ,
\end{cases}
\end{align}
 here, $\delta_y$ is the Dirac function located at $y\in\Omega$. The Green's function can be decomposed into a fundamental singularity and a regular part:
\begin{align*}
    G(x,y)=\begin{cases}\frac{1}{2\pi}(-\log|x-y|-H(x,y))&\text{if }N=2,\\
    \frac{\Gamma(N / 2)}{2(N-2) \pi^{N / 2}}(|x-y|^{2-N}-H(x,y))&\text{if }N=3.\end{cases}
\end{align*}
The regular part $H(x, y)$ is harmonic in both variables. 
In particular, when $x=y$, $H(x,x)$ is called \textit{Robin's function} of $\Omega$ at $x$. When $x\to \partial \Omega$, it holds $H(x,x)\to +\infty$. The \textit{harmonic radius} $\r_{\Omega}(x)$ of $\Omega$ at $x$ is defined by 
\begin{align*}
\r_{\Omega}(x):= \begin{cases}\exp (-H(x, x)) & \text{if }N=2, \\
H^{\frac{1}{2-N}}(x, x) & \text{if }N \geq 3.\end{cases}
\end{align*}
The subscript of $\r_\Omega$ will be dropped if the domain $\Omega$
is understood without ambiguity.

In two dimensions, the harmonic radius coincides with the \textit{conformal radius} which seems to appear for the first time in the proof of Riemann mapping theorem.  For every simply connected planar domain $\Omega$ whose complement contains at least two points, Riemann mapping theorem asserts that there is a $f$ from $\Omega$ to the unit disk $\mathbb{B}^2$, which is unique up to a rotation. The \textit{Liouville formula} says that $\r_\Omega=(1-|f|^2)/|f'|$. Here $f'$ is the complex derivative. In any dimension $N\geq 2$, the harmonic radius of $\Bn$ and $\Rnp$ are known
\begin{align*}
    \r_{\Bn}(x)=1-|x|^2,\quad\r_{\Rnp}(x)=2x_N.
\end{align*}



Here we list some basic properties of harmonic radius (see \cite{bandle1996harmonic}).
\begin{lemma}\label{lem:r-property} For domains with Lipschitz boundary, the harmonic radius has the following properties.
\begin{enumerate}
\item \textnormal{(Positivity)} For $x\in \Omega$, we have $\r_\Omega(x)>0$ and $\r_\Omega$ is smooth.
\item \textnormal{(Monotonicity)} If $\Omega \subset \widetilde \Omega$, then $\r_\Omega(x)\leq \r_{\widetilde \Omega}(x)$ for any $x\in \Omega$.
\item For $x\in\Omega$, $\delta_\Omega(x)\leq \r_\Omega(x)\leq \text{diam}(\Omega)/2$.
\item If $\Omega$ is convex, then $\r_\Omega\leq 2\delta_\Omega$. If $\r_\Omega=2\delta_\Omega$, then $\Omega$ is a half-space. 
\end{enumerate}
\end{lemma}

\begin{lemma}
     Suppose that $f:\Omega\to \widetilde \Omega$ is a conformal equivalence, then the following transformation rules hold
    \begin{align}
     \Delta \tilde{u} \circ f & =\left|f^{\prime}\right|^{-\frac{N+2}{2}} \Delta\left(\left|f^{\prime}\right|^{\frac{N-2}{2}} u\right),\label{Lap_change} \\
     \tilde{\r} \circ f & =\left|f^{\prime}\right| \r \label{r_change},
    \end{align}
where $|f'|:=|\det Df|^{\frac{1}{N}}$.
\end{lemma}

\section{Proof of conformal invariance and consequence on planar domains}
In this section, we give the proof of Proposition \ref{prop:main-cf} and Theorem \ref{thm:planar}. First, we have the following equalities for conformal equivalence.
\begin{lemma}\label{lem:f}
Suppose that $f:\Omega\to \widetilde \Omega$ is a conformal equivalence. For any $u\in C^\infty_0(\Omega)$, define $\tilde u\in C^\infty_0(\widetilde{\Omega})$ such that $\tilde u\circ f=|f'|^{\frac{2-N}{2}}u$. Then 
\begin{align}
    \int_{\Omega}|\nabla u|^2 dx&=\int_{\widetilde{\Omega}}|\nabla\tilde u|^2d\tilde x,\label{nablau=}\\
    \int_{\Omega} \frac{ |u|^{p}}{\r^s}dx&=\int_{\widetilde{\Omega}}\frac{|\tilde u|^{p}}{\tilde \r^s}d\tilde x.\label{rnu=}
\end{align}
\end{lemma}
\begin{proof} 
Using integration by parts, for $\tilde u\in C^\infty_0(\Omega)$ we have 
\begin{align*}
    \int_{\widetilde\Omega} |\nabla \tilde u|^2d\tilde x=-\int_{\widetilde\Omega}\tilde u\Delta \tilde u d\tilde x.
\end{align*}
Making a change of variables, $d\tilde x=|f'|^N dx$, using \eqref{Lap_change}, we have 
\begin{align*}
    \int_{\widetilde\Omega}\tilde u\Delta \tilde u d\tilde x=\int_{\Omega} |f'|^{\frac{2-N}{2}}u\cdot |f'|^{-\frac{N+2}{2}}\Delta u \cdot |f'|^Ndx=\int_\Omega u\Delta udx.  
\end{align*}
Using \eqref{r_change}, one has 
\begin{align*}
\int_{\widetilde\Omega}  \frac{|\tilde u|^{p}}{\tilde \r^s}d\tilde x=\int_{\Omega}\frac{\left||f'|^{\frac{2-N}{2}}u\right|^{p}}{(|f'|\r)^s}|f'|^Ndx&= \int_{\Omega}\frac{|u|^{p}}{\r^s}|f'|^{-s+\frac{2-N}{2}p+N}dx=\int_{\Omega}\frac{|u|^{p}}{\r^s}dx.
\end{align*}
The proof is complete.
\end{proof}

Now, we give the proof of Proposition \ref{prop:main-cf}.

\begin{proof}[Proof of Proposition \ref{prop:main-cf}]

Define the quotient
\begin{align}\label{def:Ju}
    J^{\Omega}_p[u]=\frac{\int_{\Omega}|\nabla u|^2dx}{\left(\int_{\Omega}\r_\Omega^{-s}|u|^{p}dx\right)^{2/p}}.
\end{align}
Suppose $\Omega$ and $\widetilde\Omega$ are conformally equivalent, Lemma \ref{lem:f} implies that $J^\Omega_p[u]=J^{\tilde \Omega}_p[\tilde u]$ for any $u\in C^\infty_0(\Omega)$ and its counterparts $\tilde u\in C^\infty_0(\Omega)$. Thus $\sigma_p(\Omega)\geq \sigma_p(\widetilde\Omega)$. The reverse inequality holds for the same reason.  Therefore, $\sigma_p(\Omega)=\sigma_p(\widetilde\Omega)$.

To prove the second part, we only need to show if $\mu_p(\Omega)$ is attained for some convex $\Omega$ with $\sigma_p(\Omega)=\sigma_p(\Rnp)$, then it must be a half-space. Since $\Omega$ is convex, it holds $\r_\Omega\leq 2\delta_\Omega$ according to (4) of Lemma \ref{lem:r-property}, which implies that 
\begin{align}\label{Jleq}
J^{\Omega}_p[u]=\frac{\int_{\Omega}|\nabla u|^2dx}{\left(\int_{\Omega}\r_\Omega^{-s}|u|^{p}dx\right)^{2/p}}\leq 2^{\frac{2s}{p}}\frac{\int_{\Omega}|\nabla u|^2dx}{\left(\int_{\Omega}\delta_\Omega^{-s}|u|^{p}dx\right)^{2/p}},
\end{align}
here we have used the fact that $s\geq 0$. Taking the infimum on $u\in C^\infty_0(\Omega)$, it leads to $\sigma_p(\Omega)\leq 2^{\frac{2s}{p}} \mu_p(\Omega)$. By our assumption, (2) of Theorem \ref{thm:WangZhu2} and the fact that $\sigma_p(\Rnp)=2^{\frac{2s}{p}}\mu_p(\Rnp)$, we have 
\begin{align*}
    \sigma_p(\Rnp)=\sigma_p(\Omega)\leq 2^{\frac{2s}{p}}\mu_p(\Omega)\leq 2^{\frac{2s}{p}}\mu_p(\Rnp)=\sigma_p(\Rnp).
\end{align*}
Thus all inequalities must be equalities. Thus if $\mu_p(\Omega)$ is achieved by some $u\in H^1_0(\Omega)$, then $\sigma_p(\Omega)$ can also be achieved by this $u$.  Thus \eqref{Jleq} implies that $\r(x)=2\delta(x)$ for any $x\in \Omega$. This implies that $\Omega$ must be a half-space by (4) of Lemma \ref{lem:r-property}. 
\end{proof}

Next, we give the proof of Theorem \ref{thm:planar}.

\begin{proof}[Proof of Theorem \ref{thm:planar}]
To prove (1), Riemann mapping theorem implies that $\sigma_p(\Omega)=\sigma_p(\mathbb{R}^2_+)$ for any simply connected Lipschitz hyperbolic region $\Omega$. Then we only need to show $\sigma_2(\mathbb{R}^2_+)=1$. In the planar case, one has $s=2$. Note that $\sigma_p(\Rnp)=2^{\frac{2s}{p}}\mu_p(\Rnp)$ and $\mu_2(\Rnp)=1/4$, the conclusion is obvious.

To prove (2).  Since a convex planar domain with non-empty boundary must be simply connected Lipschitz hyperbolic region, then part (1) shows that $\sigma_p(\Omega)=\sigma_p(\mathbb{R}^2_+)$. It follows from Proposition \ref{prop:main-cf} that $\Omega$ is a half-plane. When $p=2$,  $\mu_2(\Omega)$ is the Hardy constant $\mu^\text{H}(\Omega)$, which is not attained for half-plane. If $p>2$, Theorem \ref{thm:WangZhu2} implies that $\mu_p(\mathbb{R}^2_+)$ is achieved. 
\end{proof}

\section{Proof of Theorem \ref{thm:H<0}}\label{sec:H<0}
Fixing any point $P$ on $\partial \Omega$, we assume that $\partial\Omega$ is $C^2$ near to it. Throughout this section, let $P$ be the origin $0\in \Rn$ for simplicity. One has the so-called \textit{Fermi coordinates} on a neighborhood $\mathcal{O}$ of the origin in $\Omega$, say $\Phi\in C^2$, which maps  $B_{2\rho_0}^{+}=\left\{(x', x_N):|x|<2\rho_0, x_N>0\right\}$  to $\mathcal{O}$, for instance, see \cite[section 3]{escobar1992conformal}. More precisely, $y=\Phi(x)=\left(\Phi_{1}(x), \cdots, \Phi_{N}(x)\right)$ for $x\in B_{2\rho_0}^{+}$ form a set of coordinates near $0\in \partial \Omega$ with the following two properties. First, $\left(x_1, \cdots, x_{N-1}\right)$ is normal coordinates on $\partial \Omega$ at the point $0$. Second, the geodesic leaving from $\left(x_1, \cdots, x_{N-1}\right)$ in the orthogonal direction to $\partial \Omega$ is parameterized by $x_N$. In these coordinates, $(\mathcal{O},g_E)$ is isometric to $(B_{2\rho_0}, g)$ where 
$$
g:=\Phi^*g_E=d x_N^2+g_{i j}(x',x_N) d x_i d x_j
$$
where the index $i,j$ run from $1$ to $N-1$. Repeated index means summation. We have the following expansion formula of the metric $g$ from \cite{escobar1992conformal}.
\begin{lemma}\label{lem:exp-gE} Let $(g^{ij})$ be the inverse of the metric $(g_{ij})$. For $|x|$ small, we have
\begin{enumerate}
    \item $g^{i j}(x',x_N)=\delta^{i j}+2 h_{i j} x_N+o(|x|),$
    \item $\sqrt{\operatorname{det}\left(g_{i j}\right)}(x,x_N)=1-H x_N+o(|x|)$,
\end{enumerate}
where $\left\{h_{i j}\right\}$ is the second fundamental form at 0, $H=\text{tr}\,h$ is the mean curvature at 0.
\end{lemma}

 Suppose that $\varphi(\rho)$ is a smooth decreasing function of $\rho$, which satisfies $\varphi(\rho)=1$ for $\rho \leq \rho_0$,  $\varphi(\rho)=0$ for $\rho \geqslant 2 \rho_0$ and $\left|\varphi^{\prime}(\rho)\right| \leq 2 \rho_0^{-1}$ for $\rho_0 \leq \rho \leq 2 \rho_0$. Let 
 \begin{align}\label{def:u}
u(x)=\varphi(|x|)U_{\varepsilon}(x),
\end{align}
where $U_\varepsilon(x)=\varepsilon^{1-\frac{N}{2}}U(\frac{x}{\varepsilon})$ and $U$ is the extremal function of $\mu_p(\Rnp)$ defined in Theorem \ref{thm:WangZhu}. Via Fermi coordinates, one can define $\tilde u(y)= u(\Phi^{-1}(y))$, which is well defined for any $y\in \Omega$ because of the support of the $\varphi$. 

\begin{lemma}\label{lem:F-1}
Suppose that $\rho_0$ is a fixed small number and $u$ is defined through \eqref{def:u}. For $\varepsilon$ sufficiently small, we have
\begin{align}\label{def:I1}
\int_{\Omega}|\nabla \tilde u|^2_{g_E}dvol_{g_E}=\int_{B_{2\rho_0}^{+}}|\nabla u|^2_{g}dvol_{g}=\int_{\Rnp}|\nabla U|^2dx+\varepsilon H I_1+o(\varepsilon),
\end{align}
where $I_1=\int_{\Rnp}[-x_N|\nabla U|^2+2x_N(\partial_{x_1}U)^2]dx$.
\end{lemma}

\begin{proof}
The first equal sign is obvious because $\Phi$ is an isomorphism. Let us first compute the integral in $B_{\rho_0}^+$.
From the expression of $g=\Phi^*g_E$ and Lemma \ref{lem:exp-gE}, we have
\begin{align}
\int_{B_{\rho_0}^{+}}|\nabla u|_{g}^2 dvol_{g} & =\int_{B_{\rho_0}^{+}}\left(g^{i j} u_i u_j+u_N^2\right) dvol_{g} =\int_{B_{\rho_0}^{+}}\left(g^{i j} u_i u_j+u_N^2\right) \sqrt{\operatorname{det}(g)} dx\notag\\
&=\int_{B_{\rho_0}^{+}}|\nabla u|^2 \sqrt{\operatorname{det}(g)} dx+2 h_{i j} \int_{B_{\rho_0}^{+}}x_Nu_i u_j \label{energy-1}\sqrt{\det(g)}dx+o(e),
\end{align}
where $u_i=\partial_{x_i} u$, $|\nabla u|^2=u_1^2+\cdots+u_{N-1}^2+u_N^2$ and 
\begin{align*}
e:=\int_{B_{\rho_0}^+}|x||\nabla u|^2 \sqrt{\det(g)} dx.
\end{align*}
For the first term on the right hand side of \eqref{energy-1}, we apply Lemma \ref{lem:exp-gE} again to get 
\begin{align}
\begin{split}\label{1-1}
\int_{B_{\rho_0}^{+}}|\nabla u|^2 \sqrt{\operatorname{det}(g)} dx  = &\,\int_{B_{\rho_0}^{+}}|\nabla u|^2 dx-H \int_{B_{\rho_0}^{+}} x_N|\nabla u|^2 dx+o(e).
\end{split}
\end{align}
 To deal with the first term on the right-hand side of the above equality, 
using the decay estimates of $U$ in \eqref{u-bound}, we have 
\begin{align}
\begin{split}\label{1-1-1}
    \int_{B_{\rho_0}^{+}}|\nabla u|^2 dx=\int_{B_{\rho_0/\varepsilon}^+}|\nabla U|^2dx&=\int_{\Rnp}|\nabla U|^2dx-\int_{\Rnp\setminus B_{\rho_0/\varepsilon}^+}|\nabla U|^2dx\\
    &=\int_{\Rnp}|\nabla U|^2dx+O(\varepsilon^N\rho_0^{-N}).
    \end{split}
\end{align}
Similarly, we estimate the second term on the right-hand side of \eqref{1-1} as the following
\begin{align*}
    \int_{B_{\rho_0}^{+}}x_N|\nabla u|^2 dx=\varepsilon\int_{B_{\rho_0/\varepsilon}^+}x_N|\nabla U|^2dx&=\varepsilon\int_{\Rnp}x_N|\nabla U|^2dx+O(\varepsilon^N\rho_0^{1-N}).
\end{align*}
To deal with the last term on the right-hand side of \eqref{1-1}, we notice that on $B_{\rho_0}^{+}$,
$$
|\nabla u|^2(x) \leq C \varepsilon^{-N}\left(\frac{|x|}{\varepsilon}+1\right)^{-2 N},
$$
which leads to 
\begin{align*}
\begin{split}
 e&\leq C\varepsilon \int_{B_{\rho_0 / \varepsilon}^{+}}|x|(1+|x|)^{-2 N} dx 
 \leq C  \varepsilon.
\end{split}
\end{align*}
Now we insert the previous three estimates to \eqref{1-1} and  $\varepsilon^N\rho_0^{-N}=o(\varepsilon)$ to obtain
\begin{align}\label{1-final}
\begin{split}
\int_{B_{\rho_0}^{+}}|\nabla u|^2 \sqrt{\operatorname{det}(g)} dx  =&\  \int_{\Rnp} |\nabla U|^2 dx- \varepsilon H  \int_{\Rnp} x_N|\nabla U|^2 dx +o(\varepsilon).
\end{split}
\end{align}
For the second integral on the right-hand side in \eqref{energy-1}, we estimate  
\begin{align*}
\int_{B_{\rho_0}^{+}} x_N u_i u_j \sqrt{\operatorname{det}(g)} dx &= \varepsilon\int_{B_{\rho_0/\varepsilon}^{+}} x_N\,U_iU_j dx+ o(\varepsilon)=\varepsilon \int_{\Rnp} x_N\, U_i U_j dx+o(\varepsilon).
\end{align*}
The symmetries of the half-ball and $U$ imply
\begin{align}\label{2-final}
h_{ij}\int_{B_{\rho_0}^{+}} x_N u_i u_j \sqrt{\operatorname{det}(g)} dx =\varepsilon H \int_{\Rnp} x_N (\partial_{x_1}U)^2 dx+O(\varepsilon).
\end{align}
Consolidating the results in \eqref{1-final} and \eqref{2-final}, we have a bound of \eqref{energy-1},
\begin{align}
\begin{split}\label{inner-energy}
\int_{B_{\rho_0}^+}|\nabla u|_{g}^2 dvol_{g} & = \int_{\Rnp} |\nabla U|^2 dx- \varepsilon H  \int_{\Rnp} x_N|\nabla U|^2 dx \\
&\quad + 2\varepsilon H\int_{\Rnp}x_N(\partial_{x_1}U)^2dx+o(\varepsilon).
\end{split}
\end{align}
On $B_{2 \rho_0}^{+}-B_{\rho_0}^{+}$, we observe that
\begin{align*}
    |\nabla u|_{g}^2 \leq C|\nabla u|^2 \leqslant C\left(\varphi^2  \cdot\left|\nabla U_{\varepsilon}\right|^2+|\nabla \varphi|^2  U_{\varepsilon}^2\right).
\end{align*}
Using the estimates of $U$ and the facts that $0 \leqslant \varphi \leqslant 1$ and $\left|\nabla \varphi\right| \leqslant 2 / \rho_0$, we get
\begin{align}\label{ring-energy}
\int_{B_{2 \rho_0}^{+}-B_{\rho_0}^{+}}|\nabla u|_{g}^2 dvol_{g} \leqslant C \left(\frac{\rho_0}{\varepsilon}\right)^{-N}=C  \varepsilon^N  \rho_0^{-N}=o(\varepsilon).
\end{align}
Combining \eqref{inner-energy}, \eqref{ring-energy}, we have 
\begin{align*}
    \int_{\Omega}|\nabla u|_{g}^2 dvol_{g}= \int_{\Rnp}|\nabla U|^2dx+\varepsilon H I_1+o(\varepsilon),
\end{align*}
where $I_1$ is defined in \eqref{def:I1}. The proof is complete.
\end{proof}

\begin{lemma}\label{lem:F-2}
Suppose that $\rho_0$ is a fixed small number and $u$ is defined through \eqref{def:u}. For $\varepsilon$ sufficiently small, one has 
\begin{align}\label{def:I2}
\int_{\Omega}\frac{ \tilde u^{p}}{\delta_\Omega^s}dvol_{g_E}=\int_{B_{2\rho_0}^+}\frac{u^p}{x_N^s}dvol_g=\int_{\Rnp}x_N^{-s}U^{p}dx-\varepsilon HI_2+o(\varepsilon),
\end{align}
where $I_2=\int_{\Rnp}x_N^{-s+1}U^pdx$.
\end{lemma}
\begin{proof}
In the Fermi coordinates, the distance function $\delta_\Omega=x_N$. This implies the first equal sign in \eqref{def:I2}.
Applying Lemma \ref{lem:exp-gE}, we get 
\begin{align}\label{energy-2}
\begin{split}
\int_{B_{\rho_0}^+} \frac{u^{p}}{x_N^s} dvol_{g}
=&  \int_{B_{\rho_0}^{+}} \frac{u^{p}}{x_N^s} \sqrt{\operatorname{det}(g)} dx \\
= &\int_{B_{\rho_0 / \varepsilon}^{+}} x_N^{-s} U^{p} dx-\varepsilon H \int_{B_{\rho_0 / \varepsilon}^{+}} x_N^{-s+1} U^{p} dx-o\left(\varepsilon\int_{B_{\rho_0 / \varepsilon}^{+}} |x|^{1-s} U^{p} dx\right).
\end{split}
\end{align}
The last term can be estimated by applying \eqref{u-bound},
\begin{align*}
\begin{split}
\int_{B_{\rho_0 / \varepsilon}^{+}} |x|^{1-s} U^{p} dx&\leq C\int_{B_{\rho_0 / \varepsilon}^{+}}(1+|x|)^{1-s-(N-1)p}dx\leq C\int_0^{\rho_0/\varepsilon}(1+r)^{-\frac{N}{2}p}dr\leq C,
\end{split}
\end{align*}
where we have used the fact $p>2$.
Similarly, for the left terms on the right hand side of \eqref{energy-2}, we have 
\begin{align*}
    \int_{B_{\rho_0 / \varepsilon}^{+}} x_N^{-s} U^{p} dx=&\int_{\Rnp} x_N^{-s} U^{p} dx-\int_{\Rnp\backslash B_{\rho_0 / \varepsilon}^{+}} x_N^{-s} U^{p} dx\\
    =&\int_{\Rnp} x_N^{-s} U^{p} dx+O \left(\varepsilon^{\frac{Np}{2}}\rho_0^{-\frac{Np}{2}}\right),
\end{align*}
and 
\begin{align*}
    \int_{B_{\rho_0 / \varepsilon}^{+}} x_N^{-s+1} U^{p} dx=\int_{\Rnp} x_N^{-s+1} U^{p} dx+O\left(\varepsilon^{\frac{Np}{2}-1}\rho_0^{1-\frac{Np}{2}}\right).
\end{align*}
On $B_{2 \rho_0}^{+}-B_{\rho_0}^{+}$, using \eqref{u-bound}, we observe that
\begin{align*}
\begin{split}
\int_{B_{2 \rho_0}^{+}-B_{\rho_0}^{+}} \frac{u^{p}}{x_N^s} dvol_{g} \leqslant C  \int_{B_{2 \rho_0}^{+}-B_{\rho_0}^{+}} x_N^{-s} U_{\varepsilon}^{p} dx&=C  \int_{B_{2 \rho_{0 }/\varepsilon}^{+}-B_{\rho_0 / \varepsilon}^{+}} x_N^{-s}(1+|x|)^{(1-N)p} dx \\
& =C \left(\frac{\rho_0}{\varepsilon}\right)^{-\frac{Np}{2}}=o(\varepsilon).  
\end{split}
\end{align*}
Back to \eqref{energy-2}, we have 
\begin{align*}
\int_{B_{2\rho_0}^+}\frac{u^{p}}{x_N^s}dvol_{g}= \int_{\Rnp} x_N^{-s}U^{p}dx-\varepsilon HI_2+o(\varepsilon),
\end{align*}
where $I_2$ is defined in \eqref{def:I2}. The proof is complete.
\end{proof}

We have the following Pohozaev-type identities.

\begin{lemma}\label{lem:Poh}
Suppose $U$ is defined in Theorem \ref{thm:WangZhu}, $I_1$ is defined in \eqref{def:I1}, and $I_2$ is defined in \eqref{def:I2}. One has
\begin{align*}
I_1+\frac{2(N-1)-s}{(N-1)p}I_2=0.
\end{align*}
\end{lemma}
\begin{proof}
Recall that $U$ satisfies 
\begin{align}\label{U-eqn}
    -\Delta U=x_N^{-s}U^{p-1},\quad \text{in }\Rnp.
\end{align}
Multiplying equation \eqref{U-eqn} by $x_NU$ and integrating by parts, we have
\begin{align}
\begin{split}\label{tv-1}
\int_{\Rnp}x_N^{-s+1}U^{p}dx&=\int_{\Rnp} \nabla U\cdot \nabla (x_NU)dx=\int_{\Rnp}x_N|\nabla U|^2dx.
\end{split}
\end{align}
Note that $x_NU$ is not in $L^2(\Rnp)$ when $N=2$. To justify the integration by parts, one has to multiply a cut-off function of $x_NU$ or do the integration by parts on compact domains, and then take the limit. We omit the details here. The above equality means $I_2=\int_{\Rnp}x_N|\nabla U|^2dx$.

Multiplying equation \eqref{U-eqn} by $x_N^2\partial_{x_N}U$ and integrating by parts on both sides, we have
\begin{align*}
-\int_{\Rnp}x_N^2\partial_{x_N}U\Delta U  dx&=\int_{\Rnp}\nabla U\cdot \nabla (x_N^2\partial_{x_N} U)dx\\
&=\frac{1}{2}\int_{\Rnp}x_N^2\partial_{x_N}|\nabla U|^2dx+2\int_{\Rnp}x_N(\partial_{x_N} U)^2dx\\
&=-\int_{\Rnp}x_N|\nabla U|^2dx+2\int_{\Rnp}x_N(\partial_{x_N} U)^2dx\\
&=I_2-2(N-1)\int_{\Rnp}x_N(\partial_{x_1}U)^2dx,
\end{align*}
and 
\begin{align*}
\int_{\Rnp}x_N^{-s}U^{p-1}\cdot x_N^2\partial_{x_N}Udx=\frac{1}{p}\int_{\Rnp}x_N^{2-s}\partial_{x_N} U^{p}dx=\frac{s-2}{p}I_2.
\end{align*}
Putting the above two identities together, we obtain
$$2\int_{\Rnp} x_N(\partial_{x_1}U)^2dx=\frac{p+2-s}{(N-1)p}I_2.$$
Inserting this identity with equation \eqref{tv-1} to the expression of $I_1$ in \eqref{def:I1}, one readily have
\begin{align*}
    I_1=-I_2+\frac{p+2-s}{(N-1)p}I_2=\frac{-2(N-1)+s}{(N-1)p}I_2.
\end{align*}
\end{proof}

Now we can prove Theorem \ref{thm:H<0}.
\begin{proof}[Proof of Theorem \ref{thm:H<0}]
Suppose the origin is on $\partial \Omega$ such that its mean curvature is negative. Set up the Fermi coordinates $\Phi$ near the origin. Define $u$ to \eqref{def:u} and $\tilde u=u\circ \Phi^{-1}$. 
We shall show $\int_{\Omega}|\nabla \tilde u|_{g_E}^2dvol_{g_E}<\mu_p(\Rnp)(\int_{\Omega} \delta^{-s} \tilde u^{p}dvol_{g_E})^{\frac{2}{p}}$ if $\varepsilon$ is chosen small enough.

Using Lemma \ref{lem:F-1} and Lemma \ref{lem:F-2}, we need to show
\begin{align*}
\int_{\Rnp}|\nabla U|^2dx+\varepsilon HI_1+o(\varepsilon)<\mu_p(\Rnp)\left(\int_{\Rnp} x_N^{-s}U^{p}dx-\varepsilon HI_2+o(\varepsilon)\right)^{\frac{2}{p}}.
\end{align*}
Since $U$ satisfies \eqref{U-eqn}, then $\int_{\Rnp}|\nabla U|^2dx=\int_{\Rnp} x_N^{-s}U^{p}dx=\left(\mu_p(\Rnp)\right)^{\frac{p}{p-2}}$. Then it suffices to show 
\begin{align}
HI_1< \mu_p(\Rnp)\left(\int_{\Rnp} x_N^{-s}U^{p}dx\right)^{\frac{2-p}{p}}\frac{2}{p}(-HI_2)=-\frac{2}{p}HI_2.
\end{align}
Note that $H<0$, then we only need to show $I_1+\frac{2}{p}I_2>0$. This can be seen from the Lemma \ref{lem:Poh} and $s>0$ that
\begin{align*}
I_1+\frac{2}{p}I_2=\frac{s}{(N-1)p}I_2>0.
\end{align*}
The proof is complete.
\end{proof}

\section{Proof of Theorem \ref{thm:H=0}}
In this section, we assume $N\geq3$, $\partial \Omega$ is $C^3$ near one point,  whose mean curvature vanishes and is non-umbilic there.
First, we have a refined expansion of the metric under Fermi coordinates from \cite{escobar1992conformal}.
\begin{lemma}\label{lem:exp-gE=0} Suppose $\partial\Omega$ is $C^3$ near origin and $H=0$ at the origin. For $|x|$ small, we have
\begin{enumerate}
    \item $g^{i j}(x',x_N)=\delta_{i j}+2 h_{i j} x_N-\frac{1}{3} \bar{R}_{iklj}x_kx_l+{g^{ij}}_{,Nm}x_Nx_m+3h_{im}{h_{mj}}x_N^2+o\left(|x|^2\right)$,
    \item $\sqrt{\operatorname{det}\left(g_{i j}\right)}(x',x_N)=1-\frac12\|h\|^2x_N^2-H_ix_ix_N-\frac{1}{6}\bar R_{ij}x_ix_j+o\left(|x|^2\right)$,
\end{enumerate}
where the coefficients are all evaluated at $0$. Here $\left\{h_{i j}\right\}$ is the second fundamental form, $\|h\|^2=h_{ij}h_{ij}$, $H$ is the trace of $h$, $H_i=(\partial_{x_i}H)(0)$.  $\{\bar{R}_{iklj}\}$ is the Riemann tensor of $\partial \Omega$ and $\{\bar{R}_{ij}\}$ is its Ricci tensor.
\end{lemma}

We define
\begin{align}\label{def:u=0}
u(x)=\varphi(|x|)(U_\varepsilon+\psi_\varepsilon)(x),
\end{align}
where $\varphi$ and $U_\varepsilon$ are as before,  $\psi_\varepsilon(x)=\varepsilon^{2-\frac{N}{2}}\psi(x/\varepsilon)$ with  $\psi(x)=Ah_{ij}x_ix_jx_NZ(x)$. Here $Z$ is some function to be chosen and $A$ is some constant to be determined. 
Such type of test function has been used by \citet{marques2007conformal}. Again, via Fermi coordinates, we define $\tilde u=u\circ \Phi$ which lives on $\Omega$.

We shall assume that $Z$ has cylindrical symmetry and a similar bound as $U$ in \eqref{u-bound}. 
\begin{align}
\begin{split}\label{Zprop}
    Z(x)=Z(|x'|&,x_n)=Z(r,x_N), \\
    |Z(x)|\leq C(1+|x|)^{-1-N},&\quad|\nabla Z(x)|\leq C(1+|x|)^{-2-N}.
\end{split}
\end{align}
Here and throughout this section, we denote $r=|x'|$. Later we will choose $Z(r,x_N)=r^{-1}\partial_rU(r,x_N)$. One can see that it satisfies the above two properties in the special case $p=2+2/N$ and $p=2+4/N$ because the explicit form \eqref{U:2/n} and \eqref{U:4/n}.

 Under the assumption of $Z$, $\psi$ can be bounded as
\begin{align}\label{psi-bound}
|\psi(x)|\leq C(1+|x|)^{2-N},\quad |\nabla \psi(x)|\leq C(1+|x|)^{1-N}.
\end{align}
One can think of $\psi$ as a perturbation of $U$, because if $\rho_0$ is sufficiently small and $A$ is fixed, then $|\psi_\varepsilon|=O(r)U_\varepsilon$ by \eqref{u-bound}. Thus, in $B_{2\rho_0}^+$, one has 
\begin{align}\label{1/2U}
    \frac12\varphi(|x|)U_\varepsilon\leq u\leq 2\varphi(|x|)U_\varepsilon.
\end{align}
\begin{lemma}\label{lem:F-1=0}
Suppose that $\rho_0$ is a fixed small number and $u$ is defined as \eqref{def:u=0} with $Z$ satisfies \eqref{Zprop}. For $\varepsilon$ sufficiently small, we have
\begin{align}\label{exp-U=0}
\int_{\Omega}|\nabla \tilde u|^2_{g_E}dvol_{g_E}=\int_{B_{2\rho_0}^+}|\nabla u|_g^2dvol_g=\int_{\Rnp}|\nabla U|^2dx+\varepsilon^2 \|h\|^2 I_5+o(\varepsilon^2),
\end{align}
where 
\begin{align}\label{def:I35}
I_5&=I_3 +\frac{2A^2}{N^2-1}\int_{\Rnp} x_N^2r^4|\nabla Z|^2+\frac{8A}{N^2-1}\int_{\Rnp}x_N^2[rU_r-r^2U_{rr}]Zdr,
\end{align}
and 
\begin{align*}
I_3=\int_{\Rnp}\left[-\frac12x_N^2|\nabla U|^2+\frac16x_1^2|\nabla U|^2+3x_N^2(\partial_{x_1}U)^2\right]dx.
\end{align*}
\end{lemma}
\begin{proof}
Similar to the proof of Lemma \ref{lem:F-1}, we apply Lemma \ref{lem:exp-gE=0} to obtain 
\begin{align}
\begin{split}\label{energy-1=0}
&\int_{B_{\rho_0}^{+}}|\nabla u|_{g_E}^2 dvol_{g_E} =  \int_{B_{\rho_0}^{+}}\left(g^{i j} u_i u_j+u_N^2\right) \sqrt{\operatorname{det}(g)} dx\\
&=\int_{B_{\rho_0}^{+}}|\nabla u|^2 \sqrt{\operatorname{det}(g)} dx+2 h_{i j} \int_{B_{\rho_0}^{+}}x_Nu_i u_j \sqrt{\det(g)}dx-\frac{1}{3} \bar{R}_{iklj}\int_{B_{\rho_0}^+}x_kx_lu_iu_j\sqrt{\det(g)}dx\\&+\int_{B_{\rho_0}^+}{g^{ij}}_{,Nm}x_Nx_mu_iu_j\sqrt{\det(g)}+3h_{im}{h_{mj}}\int_{B_{\rho_0}^+}x_N^2u_iu_j\sqrt{\det(g)}dx+o(\tilde e)\\
&= J_1+J_2-J_3+J_4+J_5+o(\tilde e),
\end{split}
\end{align}
where 
\begin{align*}
\tilde e:=\int_{B_{\rho_0}^+}|x|^2|\nabla u|^2 \sqrt{\det(g)} dx.
\end{align*}

For the first term $J_1$ on the right hand side of \eqref{energy-1=0}, we apply Lemma \ref{lem:exp-gE=0} again to get 
\begin{align}
\begin{split}\label{1-1=0}
J_1= &\, \int_{B_{\rho_0}^{+}}|\nabla u|^2 \sqrt{\operatorname{det}(g)} dx\\
=&\,\int_{B_{\rho_0}^{+}}|\nabla u|^2 dx-\frac12\|h\|^2 \int_{B_{\rho_0}^{+}} x_N^2|\nabla u|^2 dx-\frac16 \bar{R}_{ij}\int_{B_{\rho_0}^+}x_ix_j|\nabla u|^2dx+o(\tilde e)
\end{split}
\end{align}
Here we have used the cylindrical symmetry of $u$ which leads to $\int_{B_{\rho_0}^+}x_ix_N|\nabla u|^2dx=0$.

For $N
\geq 3$, the first term on the right-hand side of \eqref{1-1=0} is estimated by \eqref{1-1-1} and the following
\begin{align}\label{J-1-1-1}
\begin{split}
\int_{B_{\rho_0}^{+}}|\nabla u|^2 dx
&=\int_{\Rnp} (|\nabla U|^2 +\varepsilon^2|\nabla \psi|^2)dx+O(\varepsilon^N\rho_0^{-N}),
\end{split}
\end{align}
where we claim that 
\begin{align}\label{J-1-1-2}
\int_{B_{\rho_0/\varepsilon}^+}\nabla U\cdot \nabla \psi dx=0,
\end{align}
due to the cylindrical symmetry of $U$ and $H=0$. In fact,
 since $\psi=Ah_{kl}x_kx_lx_NZ$, we have
\begin{align}
     \psi_j&=Ax_N[2h_{jl}x_lZ+h_{kl}x_kx_lx_j r^{-1}Z_r]\label{psii},\\
       \psi_N&=Ah_{kl}x_kx_l[Z+x_NZ_N],\label{psiN}
\end{align}
where $Z_r=\partial_r Z(r,x_N)$ and $Z_N=\partial_{x_N}Z(r,x_N)$. Then by $H=0$, we have
\begin{align*}
    \int_{B_{\rho_0/\varepsilon}^+}U_N\psi_N dx=Ah_{kl}\int_{B_{\rho_0/\varepsilon}^+}x_k x_l U_N[Z+x_NZ_N]dx=AH\int_{B_{\rho_0/\varepsilon}^+}x_1^2 U_N[Z+x_NZ_N]dx=0,
\end{align*}
and summing $j$ from $1$ to $N-1$,
\begin{align*}
   \int_{B_{\rho_0/\varepsilon}^+}U_j\psi_j dx=&\int_{B_{\rho_0/\varepsilon}^+}x_j r^{-1}U_r Ax_N[2h_{jl}x_lZ+h_{kl}x_kx_lx_j r^{-1}Z_r]dx \\  =&2Ah_{jl}\int_{B_{\rho_0/\varepsilon}^+}x_jx_lx_Nr^{-1}U_r Zdx+Ah_{kl}\int_{B_{\rho_0/\varepsilon}^+}x_kx_lx_NU_rZ_rdx\\
   =&2AH\int_{B_{\rho_0/\varepsilon}^+}x_1^2x_Nr^{-1}U_rZdx+AH\int_{B_{\rho_0/\varepsilon}^+}x_1^2x_NU_rZ_rdx=0.
\end{align*}
Adding the last two equalities together, we get \eqref{J-1-1-2}. Next, we compute $\int_{\Rnp}|\nabla \psi|^2dx$. By \eqref{psii} and \eqref{psiN}, we have
\begin{align*}
    \int_{\Rnp}|\nabla \psi|^2dx &=A^2\int_{\Rnp}x_N^2[4h_{jk}h_{jl}x_kx_lZ^2+h_{kl}h_{mn}x_kx_lx_mx_nZ_r^2+4h_{ji}h_{kl}x_kx_lx_ix_jr^{-1}ZZ_r]dx\\
    &\quad +A^2\int_{\Rnp}h_{kl}h_{mn}x_kx_lx_mx_n[Z^2+x_N^2Z_N^2+2x_NZZ_N]dx\\
&=A^2\int_{\Rnp}x_N^2[4h_{jk}h_{jl}x_kx_lZ^2+(h_{kl}x_kx_l)^2Z_r^2+4(h_{kl}x_kx_l)^2r^{-1}ZZ_r]dx\\
    &\quad +A^2\int_{\Rnp}(h_{kl}x_kx_l)^2[Z^2+x_N^2Z_N^2+2x_NZZ_N]dx\\  &=A^2\|h\|^2\int_{\Rnp}x_N^2\left[\frac{4}{N-1}r^2Z^2+\frac{2}{N^2-1}r^4Z_r^2+\frac{8}{N^2-1}r^3ZZ_r\right]dx\\
    &\quad +A^2\|h\|^2\int_{\Rnp}\frac{2r^4}{N^2-1}\left[Z^2+x_N^2Z_N^2+2x_NZZ_N\right]dx,
\end{align*}
where we have used the identity (see \cite[p. 390]{marques2007conformal})
\begin{align}\label{marques}
    \int_{S_r^{N-2}}(h_{ij}x_ix_j)^2=\int_{S_r^{N-2}}h_{ij}h_{kl}x_ix_jx_kx_l=\frac{2\|h\|^2}{N^2-1}\int_{S_r^{N-2}}r^4.
\end{align}
Simple integration by parts yields
\begin{align*}
    &\quad\int_{\Rnp}|\nabla \psi|^2dx\\
    &=\frac{2A^2\|h\|^2}{N^2-1}\int_{\Rnp}\left(x_N^2r^2[2(N+1)Z^2+r^2Z_r^2+4rZZ_r]+r^4[Z^2+x_N^2Z_N^2+2x_NZZ_N]\right)dx\\
    &=\frac{2A^2\|h\|^2}{N^2-1}\int_{\Rnp}\left(x_N^2r^2[2(N+1)Z^2+r^2Z_r^2-2(N+1)Z^2]+x_N^2r^4Z_N^2\right)dx\\
    &=\frac{2A^2\|h\|^2}{N^2-1}\int_{\Rnp} x_N^2r^4(Z_r^2+Z_N^2)dx=\frac{2A^2\|h\|^2}{N^2-1}\int_{\Rnp} x_N^2r^4|\nabla Z|^2dx.
\end{align*}
Back to \eqref{J-1-1-1}, we have
\begin{align}\label{J-1-1-3}
\begin{split}
\int_{B_{\rho_0}^{+}}|\nabla u|^2 dx
&=\int_{\Rnp} |\nabla U|^2 dx+\frac{2A^2\|h\|^2}{N^2-1}\varepsilon^2\int_{\Rnp}x_N^2r^4|\nabla Z|^2dx+O(\varepsilon^N\rho_0^{-N}).
\end{split}
\end{align}

Similarly, we estimate the second term on the right-hand side of \eqref{1-1=0} as the following
\begin{align*}
\begin{split}
\int_{B_{\rho_0}^{+}} x_N^2|\nabla u|^2 dx
&=\varepsilon^2\int_{B_{\rho_0/\varepsilon}^+}x_N^2|\nabla U+\varepsilon \nabla \psi|^2dx\\
&=\varepsilon^2\int_{\Rnp}x_N^2|\nabla U|^2dx+O(\varepsilon^N\rho_0^{2-N})+O\left(\varepsilon^4
\int_{B_{\rho_0/\varepsilon}^+}x_N^2| \nabla \psi|^2dx\right)\\
&=\varepsilon^2\int_{\Rnp}x_N^2|\nabla U|^2dx+\begin{cases}
O(\varepsilon^3\rho_0^{-1}), &\text{if } N=3,\\
O(\varepsilon^4\ln\frac{\rho_0}{\varepsilon}), &\text{if } N=4,\\
O(\varepsilon^4), &\text{if } N\geq 5,
\end{cases}
\end{split}
\end{align*}
 where we have used the following fact, whose verification is similar to that of \eqref{J-1-1-2},
 \begin{align*}
     \int_{\Rnp}x_N^2\nabla U\nabla \psi  dx=0.
 \end{align*}

For the third term on the right-hand side of \eqref{1-1=0}, we have
\begin{align*}
\bar{R}_{ij}\int_{B_{\rho_0}^+}x_ix_j|\nabla u|^2dx&=\bar R\varepsilon^2\int_{B_{\rho_0/\varepsilon}^+}x_1^2|\nabla U+\varepsilon \nabla \psi|^2\\
&=\bar R\varepsilon^2\int_{\Rnp}x_1^2|\nabla U|^2+O(\varepsilon^N\rho_0^{2-N})+O\left(\varepsilon^4\int_{B_{\rho_0/\varepsilon}^+}x_1^2| \nabla \psi|^2dx\right)\\
&=-\|h\|^2\varepsilon^2\int_{\Rnp}x_1^2|\nabla U|^2+O(\varepsilon^N\rho_0^{2-N})+\begin{cases}
O(\varepsilon^3\rho_0^{-1}), & \text{if }N=3,\\
O(\varepsilon^4\ln\frac{\rho_0}{\varepsilon}), & \text{if }N=4,\\
O(\varepsilon^4), & \text{if }N\geq 5.
\end{cases}
\end{align*}
Here $\bar R$ is the scalar curvature of $\partial \Omega$ at 0. By the Gauss curvature equation, since the ambient metric is Euclidean, one has $\bar R=-\|h\|^2$, for instance, see \cite[(3.22)]{escobar1992conformal}. 

To deal with the last term on the right-hand side of \eqref{1-1=0}, applying the bound of $U$ in \eqref{u-bound} and $\psi$ in \eqref{psi-bound}, we obtain that on $B_{\rho_0}^{+}$,
$$
|\nabla u|^2(x) \leq C \varepsilon^{-N}\left(\frac{|x|}{\varepsilon}+1\right)^{-2 N} 
+C\varepsilon^{2-N}\left(\frac{|x|}{\varepsilon}+1\right)^{2-2 N}\leq C(\rho_0)\varepsilon^{-N}\left(\frac{|x|}{\varepsilon}+1\right)^{-2 N}, $$
which leads to 
\begin{align*}
\begin{split}
 \tilde e\leq C\int_{B_{\rho_0}^{+}}|x|^2|\nabla u|^2 dx &\leq \varepsilon^2 \int_{B_{\rho_0 / \varepsilon}^{+}}|x|^2(1+|x|)^{-2 N} dx  \leq C \varepsilon^2
\end{split}
\end{align*}

Now we insert the previous four estimates to \eqref{1-1=0} to get
\begin{align}\label{1-final=0}
\begin{split}
J_1=&\  \int_{\Rnp} |\nabla U|^2 dx+ \frac{2A^2\|h\|^2}{N^2-1}\varepsilon^2\int_{\Rnp} x_N^2r^4|\nabla Z|^2 dx -\frac12  \|h\|^2  \varepsilon^2\int_{\Rnp} x_N^2|\nabla  U |^2 dx \\
&+\frac16 \|h\|^2\varepsilon^2   \int_{\Rnp} x_1^2 |\nabla U|^2 dx+o(\varepsilon^2)
\end{split}
\end{align}

For the second integral $J_2$ on the right-hand side of \eqref{energy-1=0}, we estimate  
\begin{align*}
\begin{split}
J_2&=2h_{ij}\int_{B_{\rho_0}^{+}} x_N u_i u_j \sqrt{\operatorname{det}(g)} dx=2h_{ij}\int_{B_{\rho_0}^{+}} x_N u_i u_j dx+o(\tilde{e}) \\
&= 2h_{ij}\int_{B_{\rho_0}^{+}} x_N\left( U_{\varepsilon}\right)_i\left( U_{\varepsilon}\right)_j dx+4h_{ij}\int_{B_{\rho_0}^+}x_N (U_\varepsilon)_i(\psi_\varepsilon)_j dx\\
&\quad +2h_{ij}\int_{B_{\rho_0}^+}x_N(\psi_\varepsilon)_i(\psi_\varepsilon)_jdx+o(\tilde{e}).
\end{split}
\end{align*}
The first term on the right-hand side of the above equality is zero because of the symmetries of the half-ball and $U$ and $H=0$. For the other two terms, we have 
\[\int_{B_{\rho_0}^+}x_N(\psi_\varepsilon)_i(\psi_\varepsilon)_jdx=\varepsilon^3\int_{B_{\rho_0/\varepsilon}^+}x_N\psi_i\psi_jdxdx=\varepsilon^3\int_{B_{\rho_0/\varepsilon}^+}x_N(1+|x|)^{2-2N}dx=o(\varepsilon^2),\]
\[\int_{B_{\rho_0}^+}x_N(U_\varepsilon)_i(\psi_\varepsilon)_jdx=\varepsilon^{2}\int_{\Rnp}x_NU_i\psi_jdx+O(\varepsilon^N\rho_0^{2-N}).\]
We recall \eqref{psii} to see that 
\begin{align*}
    h_{ij}\int_{\Rnp}x_NU_i\psi_jdx&=A h_{ij}\int_{\Rnp}x_N^2x_ir^{-1}\partial_rU[2h_{jl}x_lZ+h_{kl}x_kx_lx_j r^{-1}Z_r]dx\\
    &=A\|h\|^2\int_{\Rnp}\left[\frac{2}{N-1}x_N^2(rU_r)Z+\frac{2}{N^2-1}x_N^2r^2U_rZ_r\right]dx\\
    &=\frac{2A\|h\|^2}{N^2-1}\int_{\Rnp}x_N^2[(N+1)rU_rZ-NrU_rZ-r^2U_{rr}Z]dx\\
    &=\frac{2A\|h\|^2}{N^2-1}\int_{\Rnp}x_N^2[rU_r-r^2U_{rr}]Zdx.
\end{align*}
Therefore 
\begin{align*}
    J_2=\frac{8A\|h\|^2}{N^2-1}\varepsilon^{2}\int_{\Rnp}x_N^2[rU_r-r^2U_{rr}]Zdx+o(\varepsilon^2).
\end{align*}

For $J_3$, we have
\begin{align}
\begin{split}\label{J3}
J_3&=\frac{1}{3} \bar{R}_{iklj}\int_{B_{\rho_0}^+}x_kx_lu_iu_j\sqrt{\det(g)}dx
=\frac{1}{3} \bar{R}_{iklj}\int_{B_{\rho_0}^+}x_kx_lu_iu_jdx+o(\tilde e)\\
&=\frac{1}{3} \bar{R}_{iklj}\varepsilon^2\int_{B_{\rho_0/\varepsilon}^+}x_kx_l(U_i+\varepsilon\psi_i)(U_j+\varepsilon\psi_j)dx+o(\tilde e)\\
&=\frac{1}{3} \bar{R}_{iklj}\varepsilon^2\int_{B_{\rho_0/\varepsilon}^+}x_kx_lx_ix_jr^{-2}(\partial_r U)^2dx+o(\tilde e)=o(\tilde e)
\end{split}
\end{align}
where we have used the symmetries of the curvature tensor (i.e. $\bar{R}_{ijkl}+\bar{R}_{iklj}+\bar{R}_{iljk}=0$) and cylindrical symmetry of $U(x',x_N)=U(r,x_N)$ to write $U_i=x_ir^{-1}\partial_r U$.

Similarly, using the cylindrical symmetry of $U$, we have
\begin{align}\label{J4}
    J_4=\int_{B_{\rho_0}^+}{g^{ij}}_{,Nm}x_Nx_mu_iu_j\sqrt{\det(g)}dx=o(\tilde e).
\end{align}

For $J_5$, we have
\begin{align*}
\begin{split}
J_5&=3h_{im}{h_{mj}}\int_{B_{\rho_0}^+}x_N^2u_iu_j\sqrt{\det(g)}dx=3h_{im}{h_{mj}}\int_{B_{\rho_0}^+}x_N^2u_iu_jdx+o(\tilde e)\\
&= 3h_{im}{h_{mj}}\int_{B_{\rho_0}^+}x_N^2(U_\varepsilon)_i(U_\varepsilon)_jdx+o(\tilde e)\\
&= 3\|h\|^2\varepsilon^2\int_{\Rnp}x_N^2(\partial_{x_1}U)^2dx+o(\tilde e)+O(\varepsilon^N\rho_0^{2-N}).
\end{split}
\end{align*}

Combining the above estimates of $J_1$ to $J_5$, we have a bound of \eqref{energy-1=0},
\begin{align*}
\begin{split}
\int_{B_{\rho_0}^+}|\nabla u|_{g}^2 dvol_{g} =&\, \int_{\Rnp} |\nabla U|^2 dx+\frac{2A^2\|h\|^2}{N^2-1}\varepsilon^2\int_{\Rnp} x_N^2r^4|\nabla Z|^2dx\\
&-\frac12 \|h\|^2 \varepsilon^2  \int_{\Rnp} x_N^2|\nabla U|^2 dx +\frac16\varepsilon^2\|h\|^2   \int_{\Rnp} x_1^2|\nabla U|^2 dx\\
&+\frac{8A\|h\|^2}{N^2-1}\varepsilon^2\int_{\Rnp}x_N^2[rU_r-r^2U_{rr}]Zdx\\
&+3\|h\|^2\varepsilon^2 \int_{\Rnp}x_N^2(\partial_{x_1}U)^2dx+o(\varepsilon^2).
\end{split}
\end{align*}

On $B_{2 \rho_0}^{+}-B_{\rho_0}^{+}$, we can derive a similar estimates as \eqref{ring-energy}.
\begin{align*}
\int_{B_{2 \rho_0}^{+}-B_{\rho_0}^{+}}|\nabla u|_{g}^2 dvol_{g} \leqslant C \cdot\left(\frac{\rho_0}{\varepsilon}\right)^{-N}=C  \varepsilon^N  \rho_0^{-N}=o(\varepsilon^2).
\end{align*}
Combining the above two estimates, we have 
\begin{align*}
    \int_{B_{2\rho_0}^+}|\nabla u|_{g}^2 dvol_{g}= \int_{\Rnp}|\nabla U|^2dx+\varepsilon^2\|h\|^2 I_5+o(\varepsilon^2),
\end{align*}
where $I_5$ is defined in \eqref{def:I35}. This proves \eqref{exp-U=0}.
\end{proof}

\begin{lemma}\label{lem:F-2=0}
Suppose that $\rho_0$ is a fixed small number and $u$ is defined through \eqref{def:u} with $Z$ satisfies \eqref{Zprop}. For $\varepsilon$ sufficiently small, one has 
\begin{align}\label{exp-ubd=0}
\int_{\Omega} \frac{\tilde u^{p}}{\delta_\Omega^s}dvol_{g_E}=\int_{B_{2\rho_0}^+} \frac{u^{p}}{x_N^s}dvol_{g}=\int_{\Rnp}x_N^{-s}U^{p}dx-\varepsilon^2\|h\|^2 I_6+o(\varepsilon^2),
\end{align}
where 
\begin{align}\label{def:I46}
    I_6=I_4-p(p-1)\frac{A^2}{N^2-1}\int_{\Rnp}x_N^{2-s}r^4U^{p-2}Z^2,
\end{align}
and $I_4=\int_{\Rnp}\left[\frac12 x_N^{-s+2}U^{p}-\frac{1}{6}x_N^{-s}x_1^2U^{p}\right]dx$.
\end{lemma}

\begin{proof}
Via Fermi coordinates, the distance function $\delta_\Omega$ is just $x_N$.
Applying Lemma \ref{lem:exp-gE=0},
\begin{align*}
\begin{split}
\int_{B_{\rho_0}^+} \frac{ u^{p}}{x_N^s} dvol_{g}
=&  \int_{B_{\rho_0}^{+}} x_N^{-s} (U_{\varepsilon}+\psi_\varepsilon)^{p} \sqrt{\operatorname{det}(g)} dx \\
= &\int_{B_{\rho_0 / \varepsilon}^{+}} x_N^{-s} (U+\varepsilon\psi)^{p} dx-\frac12\varepsilon^2\|h\|^2  \int_{B_{\rho_0 / \varepsilon}^{+}} x_N^{-s+2} (U+\varepsilon\psi)^{p} dx\\
&-\varepsilon^2H_i\int_{B_{\rho_0/\varepsilon}^+}x_N^{-s+1}x_i(U+\varepsilon\psi)^{p}dx-\frac16\bar R_{ij}\varepsilon^2\int_{B_{\rho_0/\varepsilon}^+}x_N^{-s}x_ix_j(U+\varepsilon\psi)^{p}dx\\
&+o\left(\varepsilon^2\int_{B_{\rho_0 / \varepsilon}^{+}} |x|^2x_N^{-s} (U+\varepsilon\psi)^{p} dx\right)=\tilde J_1+\tilde J_2+\tilde J_3+\tilde J_4+o(\varepsilon^2),
\end{split}
\end{align*}
where the last term is estimated using \eqref{u-bound},  $p>2$ and $N\geq 3$,
\begin{align*}
\begin{split}
\int_{B_{\rho_0 / \varepsilon}^{+}} |x|^2x_N^{-s} U^{p} dx\leq C\int_{B_{\rho_0 / \varepsilon}^{+}}(1+|x|)^{2-s-(N-1)p}dx\leq C\int_0^{\rho_0/\varepsilon}(1+r)^{2-\frac{Np}{2}}dr\leq C,
\end{split}
\end{align*}
and the fact $\varepsilon\psi=O(\varepsilon|x|)U\leq \rho_0 U$ in $B_{\rho_0/\varepsilon}^+$ which follows from
 \eqref{u-bound} and \eqref{psi-bound}.
 
For $\tilde J_1$, by $(U+\varepsilon\psi)^p=U^p+p\varepsilon U^{p-1}\psi+\frac{p(p-1)}{2}\varepsilon^2 U^{p-2}\psi^2+O(\varepsilon^3 U^p)$, 
we have
\begin{align}\label{tilde-J-1}
    \tilde J_1
    &=\int_{B_{\rho_0/\varepsilon}^+}x_N^{-s}U^p+\frac{p(p-1)}{2}\varepsilon^2\int_{B_{\rho_0/\varepsilon}^+}x_N^{-s}U^{p-2}\psi^2dx+O\left(\varepsilon^3\int_{B_{\rho_0/\varepsilon}^+}x_N^{-s}|x|^3U^p\right)\nonumber\\
    &=\int_{\Rnp}x_N^{-s}U^p+\frac{p(p-1)}{2}\varepsilon^2\int_{\Rnp}x_N^{-s}U^{p-2}\psi^2dx+O(\varepsilon^3)+O(\varepsilon^{\frac{Np}{2}}\rho_0^{-\frac{Np}{2}}),
\end{align}
where we have used the following fact, which holds by the symmetry of $U$ and $\psi$ and $H=0$
\begin{align*}
    \int_{B_{\rho_0/\varepsilon}^+}x_N^{-s}U^{p-1}\psi dx=0.
\end{align*}
For the second term on the right hand of \eqref{tilde-J-1}, applying \eqref{marques}, we have
\begin{align*}
    \int_{\Rnp}x_N^{-s}U^{p-2}\psi^2dx=A^2\int_{\Rnp}x_N^{2-s}(h_{ij}x_ix_j)^2U^{p-2}Z^2dx=\frac{2A^2\|h\|^2}{N^2-1}\int_{\Rnp}x_N^{2-s}r^4U^{p-2}Z^2dx.
\end{align*}
Back to \eqref{tilde-J-1}, it holds that
\begin{align*}
    \tilde{J}_1=\int_{\Rnp}x_N^{-s}U^p+p(p-1)\frac{A^2\|h\|^2}{N^2-1}\varepsilon^2\int_{\Rnp}x_N^{2-s}r^4U^{p-2}Z^2dx+O(\varepsilon^3).
\end{align*}

For $\tilde J_2$, similarly with $\tilde{J}_1$, we have
\begin{align*}
    \tilde J_2=&-\frac12\|h\|^2\varepsilon^2\int_{B_{\rho_0/\varepsilon}^+}x_N^{-s+2}(U+\varepsilon\psi)^pdx\\
    =&-\frac12\|h\|^2\varepsilon^2\int_{B_{\rho_0/\varepsilon}^+}x_N^{-s+2}U^pdx+O\left(\varepsilon^4\int_{B_{\rho_0/\varepsilon}^+}x_N^{-s+2}U^pdx\right)\\
    =&-\frac12\|h\|^2\varepsilon^2\int_{\Rnp}x_N^{-s+2}U^pdx+O(\varepsilon^{\frac{Np}{2}}\rho_0^{2-\frac{Np}{2}})+O(\varepsilon^4).
\end{align*}

For $\tilde J_3$, since $\int_{B_{\rho_0/\varepsilon}^+}x_N^{1-s}x_i U^pdx=0$ due to the  cylindrical symmetry of $U$, we have 
\begin{align*}
\tilde J_3= O(\varepsilon^3\int_{B_{\rho_0/\varepsilon}^+}x_N^{1-s}|x|U^{p}dx)=O(\varepsilon^3).
\end{align*}
By the Gauss curvature equation, since the ambient metric is Euclidean, one has $\bar R=-\|h\|^2$. Therefore,
\begin{align*}
\tilde J_4=& -\frac16\bar R_{ij}\varepsilon^2\int_{B_{\rho_0/\varepsilon}^+}x_N^{-s}x_ix_j(U+\varepsilon\psi)^{p}dx\\
=&-\frac{1}{6}\bar R\varepsilon^2\int_{B_{\rho_0/\varepsilon}^+}x_N^{-s}x_1^2U^{p}dx+O(\varepsilon^3\int_{B_{\rho_0/\varepsilon}^+}x_N^{-s}x_1^2U^{p}dx)\\
=&\,\frac{1}{6}\|h\|^2\varepsilon^2\int_{\Rnp}x_N^{-s}x_1^2U^{p}dx+O(\varepsilon^{\frac{Np}{2}}\rho_0^{2-\frac{Np}{2}})+O(\varepsilon^3).
\end{align*}
On $B_{2 \rho_0}^{+}-B_{\rho_0}^{+}$, we observe that
\begin{align*}
\begin{split}
\int_{B_{2 \rho_0}^{+}-B_{\rho_0}^{+}} x_N^{-s} U^{p} dvol_{g_E} &\leqslant C  \int_{B_{2 \rho_0}^{+}-B_{\rho_0}^{+}} x_N^{-s} U_{\varepsilon}^{p} dx=C  \int_{B_{2 \rho_{0 }/\varepsilon}^{+}-B_{\rho_0 / \varepsilon}^{+}} x_N^{-s}(1+|x|)^{-Np} dx \\
& =C \left(\frac{\rho_0}{\varepsilon}\right)^{-\frac{Np}{2}}=C \varepsilon^{\frac{Np}{2}}  \rho_0^{-\frac{Np}{2}}=o(\varepsilon^2) 
\end{split}
\end{align*}
because $p>2$ and $N\geq3$.
Combining the inequalities of $\tilde J_1$ to $\tilde J_4$ and the above, we have 
\begin{align*}
\int_{B_{2\rho_0}^+}\frac{u^{p}}{x_N^s}dvol_{g}= \int_{\Rnp} x_N^{-s}U^{p}dx-\varepsilon^2 \|h\|^2I_6+o(\varepsilon^2),
\end{align*}
where $I_6$ is defined in \eqref{def:I46}. This completes the proof.
\end{proof}

Now we can prove Theorem \ref{thm:H=0}.
\begin{proof}[Proof of Theorem \ref{thm:H=0}]
Suppose the origin is on $\partial \Omega$ such that its mean curvature is zero and $\|h\|\neq 0$. Set up the Fermi coordinates $\Phi$ near the origin. Define $u$ as \eqref{def:u=0} and $\tilde u=u\circ \Phi^{-1}$. 
We shall show $\int_{\Omega}|\nabla \tilde u|_{g_E}^2dvol_{g_E}<\mu_p(\Rnp)(\int_{\Omega} \delta_\Omega^{-s} \tilde u^{p}dvol_{g_E})^{\frac{2}{p}}$ if $\varepsilon$ is chosen small enough.

Using Lemma \ref{lem:F-1=0} and Lemma \ref{lem:F-2=0}, we need to show
\begin{align*}
\int_{\Rnp}|\nabla U|^2dx+\varepsilon^2 \|h\|^2I_5+o(\varepsilon^2)\leq \mu_p(\Rnp)\left(\int_{\Rnp} x_N^{-s}U^{p}dx-\varepsilon^2 \|h\|^2I_6+o(\varepsilon^2)\right)^{\frac{2}{p}}.
\end{align*}
Since $U$ satisfies \eqref{U-eqn}, then $\int_{\Rnp}|\nabla U|^2dx=\int_{\Rnp} x_N^{-s}U^{p}dx=\left(\mu_p(\Rnp)\right)^{\frac{p}{p-2}}$. Then it suffices to show 
\begin{align*}
\|h\|^2I_5< \frac{2}{p}(-\|h\|^2I_6)=-\frac{2}{p}\|h\|^2I_6.
\end{align*}
Note that $\|h\|\neq 0$, then we only need to show $I_5+\frac{2}{p}I_6<0$. Unfortunately, in this case, we do not have a nice Pohozaev identity.

Choosing $Z(r,x_N)=r^{-1}\partial_rU(r,x_N)$ and using \eqref{def:I35} and \eqref{def:I46},  it leads to 
\begin{align*}
    I_5+\frac{2}{p}I_6&=I_3+\frac{2}{p}I_4+\frac{8A}{N^2-1}\int_{\Rnp}x_N^2[rU_r-r^2U_{rr}]Zdx\\
     &\quad+\frac{2A^2}{N^2-1}\int_{\Rnp}x_N^2r^4|\nabla Z|^2 dx-2(p-1)\frac{A^2}{N^2-1}\int_{\Rnp}x_N^{2-s}r^4U^{p-2}Z^2.
\end{align*}

We want to find $A$ such that $I_5+\frac{2}{p}I_6<0$. It suffices to have the discriminant $D_N>0$, where 
\begin{align*}
\begin{split}
D_N=&\frac{8}{N^2-1}\left(\int_{\Rnp}x_N^2[rU_r-r^2U_{rr}]Zdx\right)^2-(I_3+\frac{2}{p}I_4)\\
&\hspace{2cm}\times \left(\int_{\Rnp}[x_N^2r^4|\nabla Z|^2-(p-1)x_N^{2-s}r^4U^{p-2}Z^2]dx\right)\
\end{split}
\end{align*}


\begin{enumerate}
    \item If $p=2+\frac{2}{N}$, we have 
\begin{align*}
U(x',x_N)=(2N)^{\frac{N}{2}}\frac{x_N}{[(1+x_N)^2+|x'|^2]^{N/2}}
\end{align*} 
Using Mathematica \cite{sun2023}, one can compute that 
\begin{align*}
    I_3+\frac{2}{p}I_4=-\frac{\sqrt{\pi } N^N ((N-26) N-8) \Gamma \left(\frac{N+1}{2}\right)}{4 (N-2) (N-1)^2 (N+1) \Gamma \left(\frac{N}{2}+2\right)}\omega_{N-2}.
\end{align*}
Let $Z(r,x_N)=r^{-1}\partial_r U(r,x_N)$.
\begin{align*}
    D_N=\frac{\pi  N^{2 N+1} (N+2) ((N-14) N-56) \Gamma \left(\frac{N+3}{2}\right)^2}{4 (N-2)^2 \left(N^2-1\right)^3 \Gamma \left(\frac{N}{2}+2\right)^2}\omega_{N-2}^2 
\end{align*}
It is positive when $N\geq 18$. Therefore we can choose $A$ such that $I_5+\frac{2}{p}I_6<0$.

\item If $p=2+\frac{4}{N}$, we have 
\begin{align*}
U(x',x_N)=(N(N+2))^{\frac{N}{4}}\frac{x_N}{[1+x_N^2+|x'|^2]^{N/2}}.
\end{align*}
Using Mathematica, one can verify that 
\begin{align*}
    I_3+\frac{2}{p+1}I_4=\frac{\pi  2^{-N-2} (N (N+2))^{N/2} (5 N+4)}{(N-2) \left(N^2-1\right)}\omega_{N-2}
\end{align*}
It is never negative for $N\geq 3$. Let $Z=r^{-1}\partial_r U(r,x_N)$,
\begin{align*}
    D_N=-\frac{\pi ^2 2^{-2 N-3} N (N (N+2))^N (N+8) }{(N-2)^2 \left(N^2-1\right)}\omega_{N-2}^2
\end{align*}
It is never positive when $N\geq 3$. Then we could not choose $A$ in this case.
\end{enumerate}

One can play around with the Mathematica code and search for new $Z$ such that $D_N>0$. Furthermore, one can choose another perturbation $\psi$, such that $I_5+\frac{2}{p}I_6<0$.


\end{proof}

\section{Non-existence of minimizers for domains near balls}\label{sec:nearball}

 In this section, we give the proof of Theorem \ref{thm:c-ball}. First, we introduce the following lemma. 
\begin{lemma}\label{lem:comp}
Let $(M,g)$ be a $C^2$ Riemannian manifold with compact boundary. Assume the Riemann curvature of $g$ is bounded in absolute value by $K_1$. Assume $s\in (0,2)$, $p\in [2,\infty)$. Suppose that $u\in C^2(M)\cap H^1_0(M)$ satisfies 
\begin{align*}
    -\Delta_g u\leq \delta^{-s}u^{p-1} \text{and }0<u<K_2 \text{ in }M,\quad  u=0 \text{ on }\partial M
\end{align*} 
Here $\Delta_g$ is the Laplace-Beltrami operator and $\delta$ is the distance function to the boundary under metric $g$. Then for any $\alpha\in (0,1)$, there exists $C=C(K_1,K_2,\alpha)$ and $\beta=\beta(K_1,K_2,\alpha)$ such that 
\begin{align*}
    u(x)\leq C\delta(x)^\alpha, \quad \text{for}\quad x\in M_\beta
\end{align*}
where $M_\beta=\{x\in M:\delta(x)<\beta\}$.
\end{lemma}
\begin{proof}
Fix any $\alpha\in (0,1)$.  Since $s<2$, $p\geq 2$ and $0<u<K_2$, there exists $\beta=\beta(K_2, \alpha)$ such that $\delta^{-s}u^{p-1}\leq \alpha(1-\alpha)\delta^{-2}u$ in $M_\beta$. Therefore $Lu:=-\Delta_g u -\alpha(1-\alpha)\delta^{-2}u\leq 0$ in $M_\beta$. That is, $u$ is a subsolution of $L$ in $M_\beta$. By Hardy's inequality near the boundary (see \cite[Lemma 1.2]{brezis1997hardy}), we also have 
    \begin{align*}
        \limsup_{\beta\to 0}\int_{M_\beta}\delta^{-2}u^2dvol_g\leq C\limsup_{\beta\to 0}\int_{M_\beta}|\nabla u|_g^2dvol_g= 0.
    \end{align*}

    Denote $w(x)=\delta^\alpha(x)-\delta(x)/2$. Using $|\nabla\delta|_g=1$ and $\Delta_g \delta^\alpha(x)=\alpha\delta(x)^{\alpha-1}\Delta_g \delta(x)+\alpha(\alpha-1)\delta(x)^{\alpha-2}|\nabla\delta(x)|_g^2$, one has 
    \begin{align}
        Lw=-\alpha \delta(x)^{\alpha-1}\Delta_g\delta(x)+\frac12\Delta_g \delta(x)+\frac{\alpha(1-\alpha)}{2}\delta^{-1}(x).
    \end{align}
    Since $|\Delta_g \delta|$ is uniformly bounded near the boundary, then $Lw\geq 0$
    for $x\in M_\beta$ provided $\beta=\beta(K_1,K_2,\alpha)$ is small enough. That is, $w$ is a supersolution of $L$ on $M_\beta$.

    We can apply the comparison principle of $L$ (see \cite[Lemma 8]{marcus1998best}, whose proof is done on Euclidean domains but can be generalized to the manifold setting with minor modifications) to see there exists $C=C(K_1,K_2,\beta)$ such that $u(x)\leq Cw(x)$ for $x\in M_{\beta/2}$.

\end{proof}

Now, we turn to Theorem \ref{thm:c-ball}. We give the definition of $\varepsilon$-close, and give a more general version of Theorem \ref{thm:c-ball}.

\begin{definition}\label{def:C2close}
Assume $\Omega_1$ and $\Omega_2$ are bounded domains in $\Rn$. We call that they are $C^2$-diffeomorphic to each other up to the boundary if there exists a $C^2$-diffeomorphism $F:\bar\Omega_1\to \bar\Omega_2$ such that $F(\bar\Omega_1)=\bar\Omega_2$. We call that they are $\varepsilon$-close in $C^2$ sense, if after a suitable translation and rotation of $\Omega_1$, 
\[\|F-id\|_{C^2}\leq \varepsilon.\]
\end{definition}

\begin{theorem}\label{thm:o-non}
    Let $\Omega_0$ be a $C^2$ bounded domain in $\Rn$  with mean curvature $H>0$ everywhere on the boundary. Assume $N\geq 3$, $p$ and $s$ satisfy \eqref{assum-pnu}. Suppose $\mu_p(\Omega_0)$ is not attainable, then for any domain $\Omega$ which is sufficiently close to $\Omega_0$ in $C^2$-sense, $\mu_p(\Omega)$ is not attainable either.
\end{theorem}

\begin{proof}[Proof of Theorem \ref{thm:c-ball}]
The conclusion follows from Theorem \ref{thm:WangZhu2} part (4) and Theorem \ref{thm:o-non}.  
\end{proof}

\begin{proof}[Proof of Theorem \ref{thm:o-non}] The proof is similar to that given in \cite{PAN1998791}.
We argue by contradiction. Suppose there exists a sequence $\{\Omega_t\}$ of domains which converge to $\Omega_0$ in $C^2$-sense as $t\to 0$, such that $\mu_p(\Omega_t)$ is attained by $0<u_t\in H^1_0(\Omega_t)$. Then there exists $C^2$-diffeomorphism $F_t:\overline{\Omega_0}\to\overline{\Omega_t}$ with $\|F_t-id\|_{C^2}\to 0$ as $t\to 0$. After multiplying some constant, we can assume $u_t$ satisfies
\begin{align*}
\begin{cases}
\Delta u_t+\delta_t^{-s}u_t^{p-1}=0,\quad u_t>0, & \text{ in } \, \Omega_t,\\
u_t=0, & \text{ on } \, \partial \Omega_t,
\end{cases}
\end{align*}
where $\delta_t=\delta_{\Omega_t}$.

Via $F_t$, we can pull back $u_t$ and the Euclidean metric on $\Omega_t$ to $\Omega_0$, that is for $x\in \Omega_0$,
\begin{align*}
&v_t(x)=u_t(F_t(x)),\quad g_t(x)=F_t^*g_E,\\
&(g_t)_{ij}(x)=\frac{\partial F_t}{\partial x_i}\cdot\frac{\partial F_t}{\partial x_j}.
\end{align*} 
Thus $(\Omega_t,g_E)$ is equivalent to $(\Omega_0,g_t)$ viewed as a family of Riemannian manifolds with boundary.
Then the $\Delta=\Delta_{g_E}$ on $\Omega_t$ is equivalent to the Laplace-Beltrami operator $\Delta_{g_t}$ on $\Omega_0$. Due to the distance-preserving property, we shall abuse the notation $\delta_t$ to denote the distance function in $(\Omega_t,g_E)$ and $(\Omega_0,g_t)$ at the same time. 
Thus $v_t$ satisfies 
\begin{align}\label{vt-eqn}
\begin{cases}
\Delta_{g_t}v_t+\delta_t^{-s} v_t^{p-1}=0,\quad v_t>0, & \text{in }\, \Omega_0,\\
v_t=0, & \text{ on } \, \partial \Omega_0.
\end{cases}
\end{align}
Since $\|F_t-id\|_{C^2}\to 0$, then 
\begin{align*}
&\int_{\Omega_t}|\nabla u_t|^2dvol_{g_E}=\int_{\Omega_0} |\nabla v_t|_{g_t}^2dvol_{g_t}=(1+o(1))\int_{\Omega_0}|\nabla v_t|^2dvol_{g_E},\\
&\int_{\Omega_t}\delta_t^{-s}u_t^pdvol_{g_E}=\int_{\Omega_0}\delta_t^{-s}v_t^pdvol_{g_t}=(1+o(1))\int_{\Omega_0}\delta_0^{-s}v_t^p dvol_{g_E}.
\end{align*}
Here and in the following, $|\cdot|$ is the norm in Euclidean metric, $|\cdot|_{g_t}=\big((g_t)^{ij}\frac{\partial\cdot}{\partial x_i}\frac{\partial\cdot}{\partial x_j}\big)^{\frac12}$ is the norm under the metric $g_t$, $dvol_{g_t}=\sqrt{\text{det}((g_t)_{ij})}dvol_{g_E}$, and ${\color{blue}\delta_0=\delta_{\Omega_0}}$. 
Since $u_t$ is a minimizer of $\mu_p(\Omega_t)\leq \mu_p(\Rnp)$ and $\mu_p(\Omega_0)$ is not attainable, then 
\begin{align*}
\mu_p(\Rnp)=\mu_p(\Omega_0)\leq \liminf_{t\to 0}J_p^{\Omega_0}[v_t]\leq \limsup_{t\to 0}J_p^{\Omega_0}[v_t]=\limsup_{t\to 0}J_p^{\Omega_t}[u_t]\leq \mu_p(\Rnp).
\end{align*}
Here $J_p^\Omega$ is defined in \eqref{def:Ju}. It follows that 
\begin{align*}
    \lim_{t\to 0}J_p^{\Omega_0}[v_t]=\mu_p(\Omega_0)=\mu_p(\Rnp).
\end{align*}
Thus
\begin{align}
\begin{split}\label{u-v-p}
\lim_{t\to 0}\int_{\Omega_t}|\nabla u_t|^2dvol_{g_E}=\lim_{t\to 0}\int_{\Omega_0}|\nabla v_t|^2dvol_{g_E}=(\mu_p(\Rnp))^{\frac{p}{p-2}},\\
\lim_{t\to 0}\int_{\Omega_t}\delta_t^{-s}u_t^p dvol_{g_E}=\lim_{t\to 0}\int_{\Omega_0}\delta_0^{-s}v_t^pdvol_{g_E}=(\mu_p(\Rnp))^{\frac{p}{p-2}}.
\end{split}
\end{align}

The proof is lengthy, we divide it into 6 steps. 

Step 1: We claim that $\|v_t\|_{L^\infty(\Omega_0)}\to+\infty $ as $t\to 0$.

Suppose not. Then after passing to a subsequence we have that $v_t\rightharpoonup v_0$ weakly in $H^1_0(\Omega_0)$ and $v_0$ satisfies $\Delta v_0+\delta_0^{-s}v_0^{p-1}=0$ weakly in $\Omega_0$. If $v_0\not \equiv 0$, then $\mu_p(\Omega_0)$ is attained by $v_0$ for
\begin{align}\label{v0min}
J^{\Omega_0}[v_0]=\left(\int_{\Omega_0}\delta_0^{-s}v_0^pdvol_{g_E}\right)^{\frac{p-2}{p}}\leq \liminf_{t\to 0}\left(\int_{\Omega_0}\delta_0^{-s}v_t^p dvol_{g_t}\right)^{\frac{p-2}{p}}=\mu_p(\Omega_0).
\end{align}
This contradicts our assumption about $\mu_p(\Omega_0)$. 

Therefore $v_0=0$. By elliptic regularity theory, we also have $v_t\to 0$ in $C_{loc}^2(\Omega_0)$. Applying Lemma \ref{lem:comp} to $v_t$ on $(\Omega_0,g_t)$, there exist two uniform constants $\beta$ and $C$ which depend on $\sup_t\|v_t\|_{L^\infty(\Omega_0)}$ such that $|v_t(x)|\leq C \delta_t^{1/2}(x) $ for $x\in (\Omega_0)_{\beta}$. Combining the estimates in the interior and on the boundary, we have $v_t\to 0$ in $C^0(\overline{\Omega_0})$. This contradicts \eqref{u-v-p}.

\bigskip

Step 2: Let $P_t$ be a maximum point of $v_t$, such point exists because $v_t$ vanishes on the boundary. We claim that 
\begin{align}\label{d/l=1}
    \lim_{t\to 0}\frac{\delta_t(P_t)}{\lambda_t}=T\in (0,1],
\end{align}
where $ \lambda_t=(\|U\|_{L^\infty(\Rnp)}/\|v_t\|_{L^\infty(\Omega_0)})^{2/(N-2)}$ and $\delta_t(P)=\text{dist}_{g_t}(P,\partial\Omega_0)$ for $P\in \Omega_0$. In addition, $(0',T)\in\Rnp$ is one of the points where $U$ achieves the global maximum.

Define
\begin{align}\label{def:vhat}
    &\hat v_t(x)=\lambda_t^{(N-2)/2}v_t(P_t+\lambda_t x) \quad \text{ for } x\in \hat \Omega_t:=(\Omega_0-P_t)/\lambda_t,\\
    &\hat g_t(x)=\lambda_t^{-2}g_t(P_t+\lambda_t x)  \quad \text { and } \quad\hat \delta_t=\text{dist}_{\hat g_t}(\cdot,\partial\hat {\Omega}_t).\notag
\end{align}
Then $\hat g_t\to g_E$ uniformly in $C^1$ on $ \hat{\Omega}_t\cap K$ for any fixed compact set $K$.
Then $\hat v_t\in H^1_0(\hat \Omega_t)$ satisfies $\hat v_t(0)=\|U\|_{L^\infty(\Rnp)}$ where $U$ is defined in Theorem \ref{thm:WangZhu}. It follows from \eqref{vt-eqn} and \eqref{u-v-p} that $\hat v_t$ satisfies
\begin{align*}
  \begin{cases}
    \Delta_{\hat g_t} \hat v_t+\hat \delta_t^{-s}\hat v_t^{p-1}=0, \, \hat v_t>0, & \text{ in }\, \hat \Omega_t,\\
    \hat v_t=0, & \text{ on } \, \partial\hat \Omega_t,
  \end{cases}
\end{align*}
and as $t\to 0$,
 \begin{align}\label{hat-int}
    \int_{\hat \Omega_t}|\nabla \hat v_t|_{\hat g_t}^2dvol_{\hat g_t}=\int_{\Omega_0}|\nabla v_t|^2_{g_t}dvol_{g_t}=\int_{\Omega_t}|\nabla u_t|^2dvol_{g_E}\to (\mu_p(\Rnp))^{\frac{p}{p-2}},\\
    \int_{\hat \Omega_t}\hat \delta_t^{-s}\hat v_t^p dvol_{\hat g_t} = \int_{\Omega_0}\delta_t^{-s}v_t^pdvol_{g_t}=\int_{\Omega_t}\delta_t^{-s}u_t^{p}dvol_{g_E}\to (\mu_p(\Rnp))^{\frac{p}{p-2}},
\end{align}
 where $\hat \delta_t=\delta_{\hat \Omega_t}$.

If there exists a subsequence such that $\delta_t(P_t)/\lambda_t\to \infty$ as $t\to 0$, then $\hat v_t\to \hat v_0$ in $C_{loc}^2(\Rn)$, where $\hat v_0$ satisfies $\Delta \hat v_0=0$ in $\Rn$ and $\int_{\Rn}|\nabla \hat v_0|^2<\infty$. This implies $\hat v_0\equiv 0$. This contradicts with $\hat v_0(0)=\lim_{t\to 0}\hat v_t(0)=\|U\|_{L^\infty(\Rnp)}>0$.

If there exists a subsequence $\delta_t(P_t)/\lambda_t\to 0$, then after suitable rotation, $\hat \Omega_t\to \Rnp$. Since $0<\hat v_t\leq \|U\|_{L^\infty(\Rnp)}$ in $\hat \Omega_t$ and the principal curvature of $\hat \Omega_t$ is uniformly bounded, then Lemma \ref{lem:comp} asserts that $\hat v_t\leq C\hat \delta_t^{1/2}$ near $\partial \hat \Omega_t$,  where $C$ is independent of $t$. This contradicts with $\hat v_t(0)=\|U\|_{L^\infty(\Rnp)}$ when $t$ is sufficiently small.

If there exists a subsequence $\delta_t(P_t)/\lambda_t\to T$ for some $T\in (0,\infty)$, then after suitable rotation,  $\hat \Omega_t\to \mathbb{R}^N_{+,T}=\{(x',x_N):x_N>-T\}$. Then there exists $\hat v_0\in  \mathcal{D}_0^{1,2}(\mathbb{R}^N_{+,T})$ such that $\hat v_t\rightharpoonup \hat v_0$ weakly in $\mathcal{D}_0^{1,2}(\mathbb{R}^N_{+,T})$ and $\hat v_t\to \hat v_0$ in $C_{loc}^2(\mathbb{R}^N_{+,T})$, 
Similar to \eqref{v0min}, one can show that $\hat v_0$ is a minimizer of $\mu_p(\mathbb{R}^N_{+,T})$  which attains the global maximum $\|U\|_{L^\infty(\Rnp)}$ at $0$. Using Theorem \ref{thm:WangZhu}, 
one has $\hat v_0(x)=U_{-Te_N,1}(x)=U(x+Te_N)$.  Furthermore, by Lemma \ref{lem:U}, $(0',T)$ is one of the points where $U$ achieves the global maximum and $T\in (0,1]$.

We also have the following facts. For any fixed $r_t\in (0,r_0)$ such that $\lim_{t\to 0}r_t/\lambda_t=\infty$, then 
\begin{align}\label{int=0}
	\lim_{t\to 0}\int_{ B_{r_t}(P_t)\cap \Omega_0}|\nabla v_t|_{g_t}^2 dvol_{g_t}=(\mu_p(\Rnp))^{\frac{p}{p-2}},\quad \lim_{t\to 0}\int_{\Omega_0\backslash B_{r_t}(P_t)}|\nabla v_t|_{g_t}^2 dvol_{g_t}=0.
\end{align}
In Fact, since $r_t/\lambda_t\to \infty$, then for any fixed $M$, by Fatou's lemma,
\begin{align*}
    \liminf_{t\to 0}\int_{B_{r_t}(P_t)\cap \Omega_0}|\nabla v_t|_{g_t}^2 dvol_{g_t}&=\liminf_{t\to 0}\int_{B_{r_t/\lambda_t}(0)\cap \hat \Omega_t}|\nabla \hat v_t|_{\hat g_t}^2dvol_{\hat g_t}\\
    &\geq \int_{B_M(0)\cap \mathbb{R}_{+,T}^N}|\nabla \hat v_0|^2dvol_{g_E}. 
\end{align*}
We have used the fact that $\hat g_t\to g_E$ locally uniformly.
Now letting $M\to \infty$, we have the right-hand side of the above converges to $(\mu_p(\Rnp))^{\frac{p}{p-2}}$. Thus
\begin{align*}
     \liminf_{t\to 0}\int_{B_{r_t}(P_t)\cap\Omega_0}|\nabla v_t|_{g_t}^2 dvol_{g_t}\geq(\mu_p(\Rnp))^{\frac{p}{p-2}}.
\end{align*}
Combining this with \eqref{hat-int}, it readily has \eqref{int=0}.



\bigskip

Step 3: In this step, we shall decompose $v_t$ to $c_tV_{x(t),\varepsilon(t),t}+w_t$, where $c_tV_{x(t),\varepsilon(t),t}$ is the best approximation, see Claim \ref{clm:inf}, and $w_t\in H^1_0(\Omega_0)$ is a small remainder.

By the compactness of $\partial \Omega_0$ and the metric $g_t$, we can fix a small $\rho_0$ such that each $Q\in \partial\Omega_0$ has a neighborhood in $(\Omega_0,g_t)$ which contains $B_{2\rho_0}^+(0)\subset \Rn$ in the Fermi coordinates $\Phi^t_Q$ at $Q$. See Section \ref{sec:H<0} for the definition of Fermi coordinates. Here the superscript $t$ indicates the metric $g_t$.

Let $Q_t\in\partial \Omega_0$ be the point such that $\text{dist}_{g_t}(P_t,\partial \Omega_0)=\text{dist}_{g_t}(P_t,Q_t)$. For any $x\in B_{\rho_0}\cap\partial\mathbb{R}^{N}_+$, $\varepsilon>0$, we define
\begin{align*}
V_{x,\varepsilon,t}(y)&=(\varphi U_{x,\varepsilon})\circ ((\Phi_{Q_t}^t)^{-1}(y)) ,\quad y\in (\Phi_{Q_t}^t)(B_{2\rho_0}^+),\\
U_{x,\varepsilon}(z)&=\varepsilon^{1-\frac{N}{2}}U\left(\frac{z-x}{\varepsilon}\right),\quad z\in \Rnp,
\end{align*}
where $\varphi$ is the cut off function used in \eqref{def:u}. Note that $V_{x,\varepsilon,t}$ is well defined in $\Omega_0$ because of the cutoff function. 

Introduce the following norm on $(\Omega_0,g_t)$.
\begin{align*}
    &\langle\nabla f,\nabla \tilde f\rangle_{g_t}=\int_{\Omega_0}g_t(\nabla f, \nabla \tilde f )dvol_{g_t}=\int_{\Omega_0}g^{ij}_t\frac{\partial f}{\partial y_i}\frac{\partial \tilde f}{\partial y_j}\sqrt{\det(g_t)}dy,\\
   &\|\nabla f\|_{g_t}=\sqrt{\langle\nabla f,\nabla f\rangle_{g_t}}.
\end{align*}
\begin{claim} 
\begin{align}\label{claim-limit}
    \|\nabla v_t-\nabla V_{0,\lambda_t,t}\|_{g_t}\to 0,\quad \text{ as } t\to 0.
\end{align}
\end{claim}
\begin{proof}
Similar to \eqref{def:vhat}, define
\begin{align*}
    \hat V_{0,\lambda_t,t}(y)= \lambda_t^{\frac{N-2}{2}}V_{0,\lambda_t,t}(P_t+\lambda_t y).
\end{align*}
Note that $\|\nabla v_t-\nabla V_{0,\lambda_t,t}\|_{g_t}=\|\nabla \hat v_t-\nabla \hat V_{0,\lambda_t,t}\|_{\hat g_t}$. Since \eqref{d/l=1}, then $\hat V_{0,\lambda_t,t}(y)\to U_{-Te_N,1}(y)=U(y+Te_N)$ in $C_{loc}^2(\mathbb{R}_{+,T}^N)$. 
 Recall that Step 2 implies that $\hat v_t\to U_{-Te_N,1}$ in $C_{loc}^2(\mathbb{R}_{+,T}^N)$ and $\hat g_t\to g_E$ locally uniformly. Then for any compact set $K\subset \mathbb{R}_{+,T}^N$, one has 
\begin{align*}
    \lim_{t\to 0}\int_{K}|\nabla \hat v_t-\nabla \hat V_{0,\lambda_t,t}(y)|^2_{\hat g_t}dvol_{\hat g_t}=0.
\end{align*}
Given any $\epsilon>0$, we choose a compact set $K\subset \mathbb{R}_{+,T}^N$ such that $\int_{\mathbb{R}_{+,T}^N\setminus K}[|\nabla U_{-Te_N,1}|^2+U_{-Te_N,1}^2]dvol_{g_E}<\epsilon$. 
Using \eqref{hat-int} and $  \lim_{t\to 0}\int_{K}|\nabla \hat v_t|_{\hat g_t}^2dvol_{\hat g_t}=\int_{K}|\nabla U_{-Te_N,1}|^2dvol_{g_E}$, then
\begin{align*}
    \limsup_{t\to 0}\int_{\hat \Omega_t
    \setminus K}|\nabla \hat v_t|_{\hat g_t}^2dvol_{\hat g_t}<\epsilon.
\end{align*}
Direct computation shows that 
\begin{align*}
    \limsup_{t\to 0}\int_{\hat \Omega_t\setminus K}|\nabla \hat V_{0,\lambda_t,t}|_{\hat g_t}^2dvol_{\hat g_t}\leq C\epsilon+C\rho_0^{-2}\int_{\mathbb{R}_{+,1}^N\setminus K}U_{-Te_N,1}^2dvol_{g_E}\leq C(1+\rho_0^{-2})\epsilon.
\end{align*}
Combining the previous three inequalities, we have
\begin{align*}
    \limsup_{t\to 0}\|\nabla\hat v-\nabla \hat V_{0,\lambda_t,t}\|_{\hat g_t}^2\leq C(1+\rho_0^{-2})\epsilon.
\end{align*}
Since $\epsilon$ is arbitrary, then \eqref{claim-limit} is established.
\end{proof}

\begin{claim}\label{clm:inf}
If $t$ is sufficiently small, then the following infimum 
\begin{align*}
    l_t:=\inf\{\|\nabla v_t-c\nabla V_{x,\varepsilon,t}\|_{g_t}:c\in \mathbb{R}, \varepsilon\in (0,1),|x|\leq \rho_0, x\in\partial{\Rnp}\}
\end{align*}
can be achieved by some $c_tV_{x(t),\varepsilon(t),t}$ where 
\begin{align}\label{limit-123}
    c_t\to 1,\quad \frac{\varepsilon(t)}{\lambda_t}\to 1,\quad |x(t)|=o(\lambda_t). 
\end{align}
\end{claim}
\begin{proof}
For fixed $t$, let $\{C_kV_{x_k,\varepsilon_k,t}\}$ be a minimizing sequence of $l_t$,
i.e.
\begin{align*}
    \|\nabla (v_t-C_kV_{x_k,\varepsilon_k,t})\|_{g_t}\to l_t, \text{ as } k \to +\infty.
\end{align*}
Then it holds 
\begin{align}\label{bd-CV}
    C_k\|\nabla V_{x_k,\varepsilon_k,t}\|_{g_t}\leq l_t+\|\nabla v_t\|_{g_t}+o(1).
\end{align}
Since $|x_k|\leq \rho_0$ and $0<\varepsilon_k<1$, it is easy to check that $\|\nabla V_{x_k,\varepsilon_k,t}\|_{g_t}$ is bounded from below by a positive constant which does not depend on $k\geq 1$ and small $t$. Then by \eqref{bd-CV}, $C_k$ is uniformly bound for all $k\geq 1$ and small $t$. Thus there is a subsequence of $\{C_k\}$ (still denote it as $\{C_k\}$) converging to $c_t\geq 0$.

Now we apply \eqref{claim-limit} to prove that $\varepsilon_k$ does not converge to $0$ as $k\to +\infty$. By \eqref{claim-limit}, we know that $l_t\to 0$ as $t\to 0$.
Then 
\begin{align*}
    -2C_k\langle\nabla v_t,\nabla  V_{x_k,\varepsilon_k,t}\rangle_{g_t}&= \|\nabla (v_t-C_k V_{x_k,\varepsilon_k,t})\|_{g_t}-\|\nabla v_t\|_{g_t}^2-C_k^2\|\nabla V_{x_k,\varepsilon_k,t}\|_{g_t}^2\\
    &\leq  l_t^2-\|\nabla v_t\|_{g_t}^2+o(1)
    \leq -\frac12\|\nabla v_t\|_{g_t}^2.
\end{align*}
If $\varepsilon_k\to 0$, then $V_{x_k,\varepsilon_k,t}\rightharpoonup 0$ weakly in $H_0^1(\Omega_0)$ as $k\to +\infty$, hence $\langle\nabla v_t,\nabla  V_{x_k,\varepsilon_k,t}\rangle_{g_t}=o(1)$ for $k\to +\infty$ and small $t$. It follows that $\|\nabla v_t\|_{g_t}=o(1)$, which contradicts \eqref{u-v-p}. Then there is a subsequence of $\{\varepsilon_k\}$ (still denote it as $\{\varepsilon_k\}$) converges to a positive constant $\varepsilon(t)$.

Since $|x_k|\leq \rho_0$, there is a subsequence (still denote it as $\{x_k\}$) converging to $x(t)\in B_{\rho_0}\cap\partial \mathbb{R}^{N}_+$. Since $C_k\to c_t$ and $\varepsilon_k\to \varepsilon(t)>0$ as $k\to +\infty$, it is easy to check that 
\begin{align*}
	\|\nabla (C_k V_{x_k,\varepsilon_k,t}-c_t V_{x(t),\varepsilon(t),t})\|_{g_t}\to 0, \text{ as }k\to +\infty,
\end{align*}
Then $l_t$ is attained by $c_t V_{x(t),\varepsilon(t),t}.$

At last, we prove \eqref{limit-123}. From \eqref{claim-limit}, we know that
\begin{align}\label{V-V}
\|\nabla V_{0,\lambda_t,t}-c_t\nabla V_{x(t),\varepsilon(t),t}\|_{g_t}\leq l_t+\|\nabla v_t-\nabla V_{0,\lambda_t,t}\|_{g_t}\to 0, \text{ as }t\to 0.
\end{align}
Then 
\begin{align}\label{limit-c-t}
	(\mu_{p}(\mathbb{R}^{N}_+))^{\frac{p}{p-2}}
=\lim_{t\to 0}\|\nabla V_{0,\lambda_t,t}\|_{g_t}^2=\lim_{t\to 0}\|c_t\nabla V_{x(t),\varepsilon(t),t}\|_{g_t}^2=	(\mu_{p}(\mathbb{R}^{N}_+))^{\frac{p}{p-2}}\lim_{t\to 0}c_t,
\end{align}
which implies that $\lim_{t\to 0}c_t=1.$ 

Again by \eqref{V-V} and \eqref{limit-c-t}, it holds
\begin{align}\label{<V,V>}
    c_t\langle\nabla V_{0,\lambda_t,t},\nabla V_{x(t),\varepsilon(t),t}\rangle_{g_t}\to (\mu_{p}(\mathbb{R}^{N}_+))^{\frac{p}{p-2}},\  \text{ as } t\to 0.
\end{align}
On the other hand, by similar calculations in \cite{bahri1989critical} or Appendix B of \cite{figalli2020on} and applying \eqref{u-bound}, we can show that for small t, 
\begin{align*}
    \langle\nabla V_{0,\lambda_t,t},\nabla V_{x(t),\varepsilon(t),t}\rangle_{g_t}\leq C(N,\partial\Omega_0)\Big[\min\{\frac{\lambda_t}{\varepsilon(t)},\frac{\varepsilon(t)}{\lambda_t},\frac{\lambda_t\varepsilon(t)}{|x(t)|^2}\}\Big]^{\frac{N}{2}}.
\end{align*}
Then combined with \eqref{<V,V>}, it implies that $\frac{\varepsilon(t)}{\lambda_t}$ is bounded from both above and below by two positive constants and $|x|=O(\lambda_t)$. 

Let $a=\lim_{t\to 0}\frac{\lambda_t}{\varepsilon(t)}$ and $w_o=\lim_{t\to 0}\frac{x(t)}{\lambda_t}$, then by \eqref{V-V} and changing variables, we have
\begin{align*}
	\int_{\mathbb{R}^{N}_+}|\nabla U(w)-a^{\frac{N}{2}}\nabla U(a(w-w_o))|^2 dw=0,
\end{align*}
which implies that $a=1$ and $w_o=0$. Thus we have \eqref{limit-123}.
\end{proof}
 According to Claim \ref{clm:inf}, we can write 
\begin{align*}
    v_t=c_tV_{x(t),\varepsilon(t),t}+w_t
\end{align*}
for some $w_t\in H^1_0(\Omega_0)$ with $\|\nabla w_t\|_{g_t}$ is small. Moreover
$\nabla w_t$ is orthogonal to the space
\begin{align}\label{orth-w}
    \text{span}\left\{\nabla V_{x(t),\varepsilon(t),t},\nabla \left(\frac{\partial V_{x(t),\varepsilon,t}}{\partial\varepsilon}\big|_{\varepsilon(t)}\right),\nabla \left(\frac{\partial V_{x,\varepsilon(t),t}}{\partial x_1}\big|_{x(t)}\right),\cdots, \nabla \left(\frac{\partial V_{x,\varepsilon(t),t}}{\partial x_{N-1}}\big|_{x(t)}\right)\right\}
\end{align}
in the inner product $\langle\cdot,\cdot\rangle_{g_t}$.

\bigskip
Step 4: In this step, we need to establish the second variation estimate.
\begin{claim}\label{clm:non-deg}
There exists $\eta_0>0$ such that if $t$ is sufficiently small, then 
\begin{align*}
    \|\nabla w_t\|_{g_t}^2\geq (p-1+\eta_0)\int_{\Omega_0}\delta_t^{-s} V_{x(t),\varepsilon(t),t}^{p-2}w_t^2dvol_{g_t}.
\end{align*}
\end{claim}
The proof by now is standard. One can argue it by contradiction using the non-degeneracy of $U$ in the Appendix. We omit the details.

\bigskip

Step 5: For the following, we will abbreviate $V_{x(t),\varepsilon(t),t}$ as $V_t$. Multiplying \eqref{vt-eqn} by  $w_t$, integrating by parts, using the orthogonal condition \eqref{orth-w}, one has 
\begin{align}
\begin{split}\label{wt2}
    \|\nabla w_t\|_{g_t}^2&=\int_{\Omega_0} \delta_t^{-s}[c_tV_{t}+w_t]^{p-1}w_tdvol_{g_t}.
\end{split}
\end{align}
Given any $\alpha>0$, recall an inequality $|(a+b)^\alpha-a^\alpha-\alpha a^{\alpha-1}b|\leq C(|b|^\alpha+|a|^{\alpha-2}\min\{a^2,b^2\})$ for any $a>0,a+b>0$ from \cite[p.\,3]{bahri1989critical}. We have
\begin{align}
    &\int_{\Omega_0} \delta_t^{-s}[c_tV_{t}+w_t]^{p-1}w_tdvol_{g_t}-\int_{\Omega_0}\delta_t^{-s}[c_t^{p-1}V_{t}^{p-1}w_t+(p-1)c_t^{p-2}V_t^{p-2}w_t^2 ]dvol_{g_t}\notag\\
    \leq& C\int_{\Omega_0}\delta_t^{-s}\left[|w_t|^p+(c_tV_t)^{p-3}\min\{(c_tV_t)^2,w_t^2\}|w_t|\right]dvol_{g_t}\label{cross-est}\\
    \leq &C\int_{\Omega_0}\delta_t^{-s}\left[|w_t|^p+(c_tV_t)^{p-3}|w_t|^3\right]dvol_{g_t}\leq C\|\nabla w_t\|_{g_t}^{\sigma}\notag,
\end{align}
where $\sigma=\min\{3,p\}$.
In the last step we have used H\"older's inequality and 
\begin{align*}
    \int_{\Omega_0}\delta_t^{-s}w_t^{p}dvol_{g_t}\leq C\left(\int_{\Omega_0}|\nabla w_t|^2_{g_t}dvol_{g_t}\right)^p.
\end{align*}
Such inequality holds if $t$ is sufficiently small since $g_t\to g_E$ as $t\to 0$. 

Now plugging in \eqref{cross-est} to \eqref{wt2}, one has
\begin{align}\label{nw=w2}
    \|\nabla w_t\|_{g_t}^2
    =\int_{\Omega_0}\delta_t^{-s}[c_t^{p-1}V_{t}^{p-1}w_t+(p-1)c_t^{p-2}V_t^{p-2}w_t^2 ]dvol_{g_t}+O(\|\nabla w_t\|_{g_t}^{\sigma}).
\end{align}
Since $c_t\to 1$ as $t\to 0$, using Claim \ref{clm:non-deg}
\begin{align}\label{wd}
    \|\nabla w_t\|_{g_t}^2\leq C(\eta_0)\int_{\Omega_0}\delta_t^{-s}V_t^{p-1}w_tdvol_{g_t}.
\end{align}

\begin{claim}\label{clm:Vw}
If $t$ is sufficiently small, then
\begin{align}\label{Vwest}
\int_{\Omega_0} \delta_t^{-s}V_t^{p-1}w_tdvol_{g_t}=O(\varepsilon^{\min\{N-2,\,2\}}).
\end{align}
\end{claim}

\begin{proof}[Proof of Claim \ref{clm:Vw}]
Using the orthogonal condition \eqref{orth-w}, we have
\begin{align}
\begin{split}\label{Vw}
    \left|\int_{\Omega_0}\delta_t^{-s}V_t^{p-1}w_tdvol_{g_t}\right|&=\left|\int_{\Omega_0}(\Delta_{g_t}V_t+\delta_t^{-s}V_t^{p-1})w_tdvol_{g_t}\right|\\
    &\leq \left(\int_{\Omega_0}w_t^{\frac{2N}{N-2}}dvol_{g_t}\right)^{\frac{N-2}{2N}}\left(\int_{\Omega_0}|\Delta_{g_t}V_t+\delta_t^{-s}V_t^{p-1}|^{\frac{2N}{N+2}}dvol_{g_t}\right)^{\frac{N+2}{2N}}\\
    &\leq\left(\int_{\Omega_0}|\Delta_{g_t}V_t+\delta_t^{-s}V_t^{p-1}|^{\frac{2N}{N+2}}dvol_{g_t}\right)^{\frac{N+2}{2N}}\|\nabla w_t\|_{g_t},
\end{split}
\end{align}
where we have used Sobolev inequality in the last step. Observe that 
\begin{align*}
    &\quad \Delta_{g_t}V_{x,\varepsilon,t}+\delta_t^{-s}V_{x,\varepsilon,t}^{p-1}\\
    &=U_{x,\varepsilon}\Delta_{g_t}\varphi+2g_t(\nabla\varphi,\nabla U_{x,\varepsilon})+\varphi\Delta_{g_t}U_{x,\varepsilon}+\delta_t^{-s}V_{x,\varepsilon,t}^{p-1}\\
    &=U_{x,\varepsilon}\Delta_{g_t}\varphi+2g_t(\nabla\varphi,\nabla U_{x,\varepsilon})+\varphi[\Delta_{g_t}U_{x,\varepsilon}+\delta_t^{-s}U_{x,\varepsilon}^{p-1}]+\delta_t^{-s}[V_{x,\varepsilon,t}^{p-1}-\varphi U_{x,\varepsilon}^{p-1}]\\
    &=\hat J_1+\hat J_2+\hat J_3+\hat J_4.
\end{align*}
Now let us compute $\hat J_i$, $i=1,2,3,4$, respectively.
\begin{align*}
    \int_{\Omega_0}|\hat J_1|^{\frac{2N}{N+2}}dvol_{g_t}&\leq C(\rho_0)\int_{B_{2\rho_0}^+/B_{\rho_0}^+}U_{x,\varepsilon}^{\frac{2N}{N+2}}dvol_{g_E}\\
    &\leq 
    C\varepsilon^{-\frac{N(N-2)}{N+2}+N}\int_{B_{2\rho_0/ \varepsilon}^+\setminus B_{\rho_0/\varepsilon}^+}(1+|x|)^{\frac{2N(1-N)}{N+2}}dx\leq C\varepsilon^{\frac{N^2}{N+2}}.
\end{align*}
Similarly 
\begin{align*}
    \int_{\Omega_0}|\hat J_2|^{\frac{2N}{N+2}}dvol_{g_t}\leq C\int_{B_{2\rho_0}^+/B_{\rho_0}^+}|\nabla U_{x,\varepsilon}|^{\frac{2N}{N+2}}dvol_{g_t}\leq C\varepsilon^{\frac{N^2}{N+2}},\\
    \int_{\Omega_0}|\hat J_4|^{{\frac{2N}{N+2}}}dvol_{g_t}\leq C\int_{B_{2\rho_0}^+/B_{\rho_0}^+}|U_{x,\varepsilon}^{p-1}|^{\frac{2N}{N+2}}dvol_{g_t}\leq C\varepsilon^{\frac{N^2}{N+2}(p-1)}.
\end{align*}
Since $g_t(y)=g_E(y)+O(y)$ uniformly at $\Omega_0$ when $t\to 0$, then on the support of $\varphi$, 
\[\Delta_{g_t}U_{x,\varepsilon}+\delta_t^{-s}U_{x,\varepsilon}^{p-1}=O(|y||\nabla_y^2U_{x,\varepsilon}|+|\nabla_y U_{x,\varepsilon}|).\]
Consequently
\begin{align*}
\int_{\Omega_0}|\hat J_3|^{\frac{2N}{N+2}}dvol_{g_t}&\leq C\int_{B_{2\rho_0}^+}[(|y||\nabla_y^2U_{x,\varepsilon}|)^{\frac{2N}{N+2}}+|\nabla_yU_{x,\varepsilon}(y)|^{\frac{2N}{N+2}}]dvol_{g_t}\\
&\leq C\varepsilon^{\frac{2N}{N+2}}\int_{B_{2\rho_0/\varepsilon}^+}(1+|y|)^{-\frac{2N^2}{N+2}}dy\leq C\varepsilon^{\frac{2N}{N+2}}.
\end{align*}
Inserting the above four estimates to \eqref{Vw}, since $N\geq 3$, one has 
\begin{align*}
    \int_{\Omega_0}\delta_t^{-s}V_{t}^{p-1}w_tdvol_{g_t}=O(\varepsilon)\|\nabla w_t\|_{g_t}.
\end{align*}
Since we also have \eqref{wd}, then $\|\nabla w_t\|_{g_t}=O(\varepsilon)$ and consequently $\int_{\Omega_0}\delta_t^{-s}V_{t}^{p-1}w_tdvol_{g_t}=O(\varepsilon^2)$.
\end{proof}

\bigskip
Step 6: Define 
\begin{align*}
    J^t(u)=\frac{\int_{\Omega_0}|\nabla u|_{g_t}^2dvol_{g_t}}{(\int_{\Omega_0}\delta_t^{-s}|u|^pdvol_{g_t})^{2/p}} \quad \text{for } u\in H_0^1(\Omega_0).
\end{align*}
By our assumption, $v_t$ is a minimizer of $\inf_{H_0^1(\Omega_0)} J^t\leq \mu_p(\Rnp)$. However, we shall prove $J^t(v_t)>\mu_p(\Rnp)$. This is a contradiction. 

To compute $J(v_t)$, we first notice that  $\|\nabla v_t\|_{g_t}^2=c_t^2\|\nabla V_t\|_{g_t}^2+\|\nabla w_t\|^2_{g_t}$.
Second, using the inequality $|(a+b)^{\alpha}-a^\alpha-\alpha a^{\alpha-1}b-\frac12\alpha(\alpha-1)a^{\alpha-1}b|\leq C(|b|^3+|a|^{\alpha-3}\min\{a^2,b^2\})$,
\begin{align*}
    &\int_{\Omega_0}\delta_t^{-s}v_t^pdvol_{g_t}\\
    &=\int_{\Omega_0}\delta_t^{-s}\left[(c_tV_t)^p+p(c_tV_t)^{p-1}w_t+\frac12p(p-1)(c_tV_t)^{p-2}w_t^2\right] dvol_{g_t}+O(\|\nabla w_t\|_{g_t}^{\sigma})\\
    &=\int_{\Omega_0}\delta_t^{-s}(c_tV_t)^pdvol_{g_t}+\frac{p}{2}\int_{\Omega_0}\delta_t^{-s}(c_tV_t)^{p-1}w_tdvol_{g_t}+\frac{p}{2}\|\nabla w_t\|_{g_t}^2+O(\|\nabla w_t\|_{g_t}^{\sigma}),
\end{align*}
where $\sigma=\min\{p,3\}$. Here, we have used \eqref{nw=w2} in the second step. 
Inserting the expansion of $\|\nabla v_t\|_{g_t}^2$ and $\int_{\Omega_0}\delta_t^{-s}v_t^pdvol_{g_t}$ into $J^t(v_t)$, using Taylor's expansion, we have
\begin{align*}
    &J^t(v_t)=J^t(c_tV_t+w_t)\\
    &=J^t(V_t)\left[1+\frac{\|\nabla w_t\|_{g_t}^2}{\|c_t\nabla V_t\|_{g_t}^2}-(1+o(1))\frac{\int_{\Omega_0}\delta_t^{-s}(c_tV_t)^{p-1}w_tdvol_{g_t}+\|\nabla w_t\|_{g_t}^2}{\int_{\Omega_0}\delta_t^{-s}(c_tV_t)^pdvol_{g_t}}\right]+O(\|\nabla w_t\|_{g_t}^{\sigma}).
\end{align*}
Using \eqref{u-v-p}, one can see that 
\begin{align*}
    \lim_{t\to 0}\|c_t\nabla V_t\|_{g_t}^2=\lim_{t\to 0}\int_{\Omega_0}\delta_t^{-s}(c_tV_t)^pdvol_{g_t}=(\mu_p(\Rnp))^{\frac{p}{p-2}},
\end{align*}
then 
\begin{align}\label{exp-Jtvt}
    J^t(v_t)=J^t(V_t)-(\mu_p(\Rnp))^{-\frac{2}{p-2}}[1+o(1)]\int_{\Omega_0}\delta_t^{-s}V_t^{p-1}w_tdvol_{g_t}.
\end{align}
Recalling Lemma \ref{lem:F-1} and Lemma \ref{lem:F-2}, 
\begin{align}\label{JtVt}
J^t(V_t)=\mu_p(\Rnp)+\varepsilon (\mu_p(\Rnp))^{\frac{p}{p-2}}H_t(x(t))(I_1+\frac{2}{p}I_2)+o(\varepsilon),
\end{align}
where $H_t(x(t))$ is the mean curvature of  $\partial \Omega_0$ at $x(t)$ with respect to metric $g_t$.

Inserting \eqref{Vwest} and \eqref{JtVt} to \eqref{exp-Jtvt}, one obtains that 
\begin{align*}
    J^t(v_t)=\mu_p(\Rnp)+\varepsilon (\mu_p(\Rnp))^{\frac{p}{p-2}}H(x(t))(I_1+\frac{2}{p}I_2)+o(\varepsilon).
\end{align*}
Suppose $P(t)\to P_0\in \partial\Omega_0$, then $H_t(x(t))\to H(P_0)>0$, where $H(P_0)$ is the mean curvature of $\partial\Omega_0$ at $P_0$ with respect to the Euclidean metric. Since $v_t$ achieves the minimum of $J^t$, using Lemma \ref{lem:Poh}, then 
\begin{align*}
    \mu_p(\Omega_t)>\mu_p(\Rnp),
\end{align*}
when $t$ is sufficiently small. This contradicts Theorem \ref{thm:WangZhu2}.

This completes the whole proof.

\end{proof}


\appendix

\section{Non-degeneracy of the minimizer}

In this section, we assume that $U$ is the unique minimizer defined in Theorem \ref{thm:WangZhu}.

\begin{proof}[Proof of Lemma \ref{lem:U}]
(1) The cylindrical symmetry is obvious from $\tilde U$ and M\"obius transformation.

(2)  From \cite{dou2022divergent} and \cite{wang2022hardy}, we know that 
\begin{align*}
U(x)=&x_N V(x), \quad\text{ for } x\in \Rnp,\\
\tilde{U}(y)=&\frac{1-|y|^2}{2}\tilde{V}(y), \quad \text{ for } y\in \Bn,
\end{align*}
where $y=\mathcal{M}x$, $\tilde V\circ \mathcal{M}=V|\det(D\mathcal{M})|^{-\frac{1}{2}}$, and $\tilde{V}$ is the unique radial solution of
\begin{align*}
\begin{cases}
  \Delta \tilde V -\frac{4\nabla \tilde V\cdot y}{1-|y|^2} -\frac{2N \tilde V}{1-|y|^2}+\left(\frac{1-|y|^2}{2}\right)^{p-2-s}\tilde V^{p-1}=0,\,  \tilde V>0,\,& \text{ in }\Bn,\\
     \tilde  V=K, & \text{ on }\partial \Bn,
\end{cases}
\end{align*}
for only one positive constant $K$. 
Let $r=|y|$, and write $\tilde U(r)=\tilde U(y)$ and $\tilde V(r)=\tilde V(y)$ for simplicity. The regularity in \cite{dou2022divergent} shows that $\tilde V\in C^2[0,1)\cap C^{\alpha}[0,1]$ for some $\alpha\in (0,1)$. It is easy to see that $\tilde V$ is bounded in $[0,1]$,
then there exists $C>1$ such that $C^{-1}(1-r^2)\leq |\tilde U(r)|\leq C(1-r^2)$ for all $r\in [0,1]$. 
Using 
\begin{align*}
    1-r^2=\frac{4x_N}{|x+e_N|^2},
\end{align*}
and \begin{align*}
    U(x)=|\det(D\mathcal{M})|^{\frac{N-2}{2N}}\tilde U\circ \mathcal{M}(x)=\frac{2^{\frac{N-2}{2}}}{|x+e_N|^{N-2}}\tilde U(r),
\end{align*}
we obtain that  for $x\in \Rnp$,
\begin{align}\label{U-estimate}
C^{-1}x_N(|x|+1)^{-N}\leq |U(x)|\leq Cx_N(|x|+1)^{-N}.
\end{align}

 Next, we give the estimate of $|\nabla U|$. By $\tilde U(r)=\frac{1-r^2}{2}\tilde V(r)$ and $\tilde U(1)=0$, we have that $\tilde U'(1)=\lim_{r\to 1^-} \frac{-\tilde U(r)}{1-r}=-\tilde V(1)$. Then we obtain that $\tilde U\in C^1[0,1]$, which implies that $\tilde U'(r)$ is bounded in $[0,1]$, i.e. $|\nabla \tilde U|$ is bounded in $\overline{\Bn}$. By simple calculation, we have
\begin{align*}
    |\nabla U|^2=&\frac{2^{N-2}}{|x+e_N|^{2N}}\left[4|\nabla_y \tilde U|^2+(N-2)^2 \tilde U^2 |x+e_N|^2+4(N-2)\tilde U \nabla_y \tilde U \cdot (x+e_N)\right]\\
    =&\frac{2^{N}}{|x+e_N|^{2N}}|\nabla_y \tilde U|^2+\frac{(N-2)^2}{|x+e_N|^2}U^2+\frac{2^{\frac{N+2}{2}}}{|x+e_N|^{N+2}}U \nabla_y \tilde U \cdot (x+e_N).
\end{align*}
Then by \eqref{U-estimate} and the boundedness of $|\nabla \tilde U|$, we obtain that for $x\in \Rnp$,
$$|\nabla U(x)|\leq C(1+|x|)^{-N}.$$ 

(3) Using coordinates $y=(y',y_N)$ and $r=|y|$ on $\Bn$, elementary computations show that for $1\leq i\leq N-1$, 
\begin{align}
\begin{split}\label{M-dU}
    |\det(D\mathcal{M})|^{\frac{2-N}{2N}}\partial_{x_i}U&=\frac{y_i}{2r}\xi(r)\circ \mathcal{M},\\
    |\det(D\mathcal{M})|^{\frac{2-N}{2N}}\partial_{x_N}U&=\left[\frac{1+y_N}{2r}\xi(r)-\tilde U'(r)\frac{|y'|^2+(1+y_N)^2}{2r}\right]\circ \mathcal{M},\\
    |\det(D\mathcal{M})|^{\frac{2-N}{2N}}\frac{d}{d\varepsilon}\big|_{\varepsilon=1}U_\varepsilon&=\frac{y_N}{2r}\xi(r)\circ \mathcal{M}.
    \end{split}
\end{align}
where $\xi(r)=\tilde U'(r)(1-r^2)-(N-2)\tilde U(r)r$.  Since $\tilde U$ satisfies
\begin{align*}
\begin{cases}
   \tilde U''+\frac{N-1}{r} \tilde 
 U'+\frac{2^s  \tilde U^{p-1}}{(1-r^2)^s}=0,\,  \tilde U>0,\,\text{ in }(0,1),\\
     \tilde U'(0)=0, \tilde  U(1)=0,
\end{cases}
\end{align*}
we have that for for $r\in (0,1]$,
$$ \tilde U'(r)=-\frac{1}{r^{N-1}}\int_0^r\frac{2^st^{N-1}\tilde U(t)^{p-1}}{(1-t^2)^s}dt<0.$$
Thus $\xi(r)<0$ for $r\in (0,1)$.

Suppose that $U$ has a critical point $P\in\Rnp$. Consider $\mathcal{M}(P)=(y_1,\cdots, y_N)\in \Bn$. Using \eqref{M-dU}, we have $y_i=0$ for $1\leq i\leq N-1$. It suffices to consider the points $(0,y_N)\in \Bn$. Note that  
\begin{align*}
    U\circ\mathcal{M}^{-1}(0,y_N)=(\det(D\mathcal{M}))^{\frac{N-2}{2N}}\tilde U(r)=\left(\frac{(1+y_N)^2}{2}\right)^{\frac{N-2}{2}}\tilde U(r).
\end{align*}
Since $U\circ\mathcal{M}^{-1}$ is strictly increasing from $(-1,0)$, we only need to consider $y_N>0$, that is the critical points of $(1+r)^{N-2}\tilde U(r)$ on $[0,1)$. 

If $N=2$, then the only critical point is $r=0$. More precisely, $U\circ \mathcal{M}$ attains the global maximum only at the origin in $\Bn$. Equivalently, $U$ attains the global maximum only at the $(0,1)\in \mathbb{R}^2_+$.

If $N>2$, it is easy to see $(1+r)^{N-2}\tilde U(r)$ has only finitely many critical points in $[0,1)$. Equivalently, $U$ has finitely many critical points in $\Rnp$, and all of them lie on the interval $\{(0',x_N):0<x_N\leq 1\}$.



\end{proof}

Next, we prove the non-degeneracy of $U$ when $N\geq 3$, which is used the Step 4 in section \ref{sec:nearball}. We denote
$\rho=x_N^{-s}U^{p-2}$ and $L^2(\Omega, \rho dx)=L^2_\rho(\Omega)$.
\begin{align}\label{w-sob}
    \left(\int_{\Rnp}x_N^{-s}|u|^pdx\right)^{2/p}\leq C\int_{\Rnp}|\nabla u|^2dx,\quad \forall\, u\in \mathcal{D}^{1,2}_0(\Rnp). 
\end{align}
\begin{proposition}
    We have a compact embedding $\mathcal{D}^{1,2}_0(\Rnp)\hookrightarrow L^2_\rho(\Rnp)$.  
\end{proposition}
\begin{proof}
    For any Borel set $\Omega\subset\Rnp$ and $f\in \mathcal{D}^{1,2}_0(\Rnp)$, using H\"older's inequality,
    \begin{align}
        \int_{\Omega}f^2\rho dx\leq \left(\int_{\Omega}x_N^{-s}U^pdx\right)^{1-\frac{2}{p}}\left(\int_{\Omega}x_N^{-s}f^pdx\right)^{\frac{2}{p}}\leq C \int_{\Rnp}|\nabla f|^2dx.
    \end{align}
   Thus the embedding is continuous. Suppose $\{\Omega_k\}_{k=1}^\infty$ is a sequence of smooth compact enlarging domains exhausting $\Rnp$. Applying the Rellich-Kondrakov theorem, one has a compact embedding as the following.
   \begin{align}
   \mathcal{D}^{1,2}_0(\Rnp)\hookrightarrow L^2(\Omega_k,dx)\hookrightarrow  L^2_\rho(\Omega_k).
   \end{align}
   Now fix any bounded sequence $\{f_j\}_{j=1}^\infty\subset \mathcal{D}^{1,2}_0(\Rnp)$. Going to a subsequence if necessary, thanks to the above compact embedding, by a diagonal argument we can find $f\in \mathcal{D}^{1,2}_0(\Rnp)$ such that $f_j\to f$ strongly in $L^2_\rho(\Omega_k)$ for each $k$ as $j\to \infty$. We want to prove $f_j\to f$ strongly in $L^2_\rho(\Rnp)$. Using H\"older's inequality and \eqref{w-sob}
   \begin{align*}
       \limsup_{j\to \infty}\|f_j-f\|_{L^2_\rho(\Rnp)}&=\limsup_{j\to \infty}(\|f_j-f\|_{L^2_\rho(\Omega_k)}+\|f_j-f\|_{L^2_\rho(\Rnp\setminus \Omega_k)})\\
       &\leq \limsup_{j\to\infty} \left(\int_{\Rnp\setminus\Omega_k}x_N^{-s}U^pdx\right)^{1-\frac{2}{p}}\left(\int_{\Rnp\setminus\Omega_k}x_N^{-s}|f_j-f|^pdx\right)^{\frac{2}{p}}\\
       &\leq C\left(\int_{\Rnp\setminus\Omega_k}x_N^{-s}U^pdx\right)^{1-\frac{2}{p}}\sup_j \int_{\Rnp}|\nabla f_j-\nabla f|^2dx.
   \end{align*}
   Now letting $k$ go to infinity, we have the desired convergence.
\end{proof}

\begin{proposition}
   There exists a compact self-adjoint operator  $(-\rho^{-1}\Delta)^{-1}$ from $L^2_\rho(\Rnp)$ to itself.  
\end{proposition}
\begin{proof}
    For any $f\in L^2_\rho(\Rnp)$ and $\phi\in \mathcal{D}^{1,2}_0(\Rnp)$, one has 
    \begin{align*}
        \int_{\Rnp}f\phi \rho dx\leq \|f\|_{L^2_\rho(\Rnp)}\|\phi\|_{L^2_\rho(\Rnp)}\leq C\|f\|_{L^2_\rho(\Rnp)}\|\phi\|_{\mathcal{D}^{1,2}_0(\Rnp)}.
    \end{align*}
Applying Riesz theorem, we have $T(f)\in \mathcal{D}^{1,2}_0(\Rnp)$ such that 
\begin{align*}
    \int_{\Rnp}f\phi \rho dx=\int_{\Rnp}\nabla T(f)\cdot\nabla \phi dx=\int_{\Rnp}-\Delta T(f)\phi.
\end{align*}
This is equivalent to say $-\rho^{-1}\Delta T(f)=f$. It is easy to see $T:L^2_\rho(\Rnp)\to \mathcal{D}^{1,2}_0(\Rnp)$ is continuous. By the previous proposition, $T:L^2_\rho(\Rnp)\to L^2_\rho(\Rnp) $ is compact.
\end{proof}

The previous two propositions imply that the spectrum of $(-\rho^{-1}\Delta)^{-1}$ is discrete. Denote the eigenvalues of   $-\rho^{-1}\Delta$ by $\lambda_i$ as the following  
\begin{align}
    \lambda_1<\lambda_2< \lambda_3< \cdots \to +\infty.
\end{align}

The first eigenvalue is $\lambda_1=1$ and eigenspace is $\text{span}\{U\}$. This follows from that $U$ satisfies $-\Delta U=\rho U$ and it is positive in $\Rnp$.
\begin{theorem}\label{thm:non-deg}
    The second eigenvalue is $\lambda_2=p-1$, which has multiplicity $N$. The corresponding eigenfunctions consist of 
\begin{align*}
    \text{span}\left\{\partial_{x_1}U,\cdots,\partial_{x_{N-1}}U,\frac{d}{d\varepsilon}\big|_{\varepsilon=1}U_\varepsilon\right\}.
\end{align*}
\end{theorem}

Before proving this theorem, let us establish an equivalent formulation in $\Bn$. For any $f\in L^2_\rho (\Rnp)$, we define $\tilde f$ on $\Bn$ by $\tilde f\circ \mathcal{M}=f|\det(D\mathcal{M})|^{\frac{2-N}{2N}}$, where $\mathcal{M}$ is the M\"obius transformation \eqref{Mo}. Denote 
$$\tilde \rho=\left(\frac{2}{1-|y|^2}\right)^s\tilde u^{p-2}.$$ 
It is easy to see that 
 \begin{align*}
     \tilde \rho\circ \mathcal{M}=
     \rho |\det(D\mathcal{M})|^{-2/N}
 \end{align*}
 and 
 \begin{align*}
     \int_{\Bn}\tilde f^2\tilde \rho dy=\int_{\Rnp}f^2 \rho |\det(D\mathcal{M})|^{\frac{2-N}{N}-\frac{2}{N}+1}dx=\int_{\Rnp} f^2\rho dx.
 \end{align*}
Therefore $L^2(\Rnp,\rho dx)$ is isometric to $L^2(\Bn,\tilde \rho dy)$.
Moreover, 
\begin{align}\label{lapf-equ}
-\Delta_{\Rnp}f=\rho g \text{ is equivalent to }-\Delta_{\Bn} \tilde f=\tilde \rho  \tilde g.
\end{align}
Therefore, it is equivalent to studying the spectrum of $(-\tilde \rho^{-1} \Delta_{\Bn})^{-1}:L^2_{\tilde \rho}(\Bn)\to L^2_{\tilde\rho}(\Bn)$. 
\begin{proof}[Proof of Theorem \ref{thm:non-deg}] 
    Define 
    \begin{align*}
        J(u)=\frac{\int_{\Rnp}|\nabla u|^2dx}{(\int_{\Rnp}x_N^{-s}|u|^p)^{2/p}}.
    \end{align*}
    Theorem \ref{thm:WangZhu} says that $J(u)\geq J(U)$ for any $u\in \mathcal{D}^{1,2}_0(\Rnp)$. The first variation near $U$ implies that $-\Delta U=x_N^{-s}U^{p-1}=\rho U$. The second variation $J''(u)(f,f)\geq 0$ implies that 
    \begin{align*}
        \int_{\Rnp}|\nabla f|^2dx\geq (p-1)\int_{\Rnp}f^2\rho dx
    \end{align*}
    for any $f\in H^1_0(\Rnp)$ satisfies $\int_{\Rnp} fU\rho dx=0$. Thus $\lambda_2\geq p-1$. Obviously any function in the set of $\text{span}\{\partial_{x_1}U,\cdots,\partial_{x_{N-1}}U,\frac{d}{d\varepsilon}\big|_{\varepsilon=1}U_\varepsilon\}$ satisfies $-\Delta f=(p-1)\rho f$. Thus $\lambda_2=p-1$.

    It suffices to prove all the eigenfunctions corresponding to $p-1$ belong to this set. We will prove this fact for $(-\tilde \rho \Delta_{\Bn})^{-1}$. The proof is inspired by \cite[Lemma 4.1]{ambrosetti2006perturbation}. Recalling \eqref{M-dU} and \eqref{lapf-equ}, one has $\{\xi(r)r^{-1}y_i\}_{i=1}^N$ satisfy the equation 
$$-\Delta \tilde f= (p-1)\tilde \rho \tilde f.$$

Let $\Delta_r$, resp. $\Delta_{\mathbb{S}^{N-1}}$ denote the Laplace operator in radial coordinates, resp. the Laplace-Beltrami operator on $\mathbb{S}^{N-1}$. Consider the spherical harmonics $\{Y_{k,i}(\theta):k=0,1,\cdots, 1\leq i\leq N_k\}$ satisfying
    \begin{align*}
        -\Delta_{\mathbb{S}^{N-1}}Y_{k,i}=\mu_kY_{k,i}
    \end{align*}
and recall that this equation has a sequence of eigenvalues 
$$\mu_k = k(k + n - 2),\quad  k=0, 1, 2,\cdots $$
whose multiplicity is given by $N_k$. In particular, one has that $\mu_0=0$ has multiplicity 1,
and $\mu_1=N-1$ has multiplicity $N$. Suppose $Y_{k,i}$ are normalized so that they form a complete orthonormal basis of $L^2(\mathbb{S}^{N-1})$. For any $\tilde f\in L^2_{\tilde \rho}(\Bn)$ can be expressed as 
\begin{align*}
\tilde f=\sum_{k=0}^\infty\sum_{i=1}^{N_k} \phi_{k,i}(r)Y_{k,i}(\theta)\text{ where }\phi_{k,i}(r)=\int_{\mathbb{S}^{N-1}}\tilde f(r,\theta)Y_{k,i}(\theta)d\theta\in L^2([0,1],\tilde \rho(r)dr).
\end{align*}
Suppose that $\tilde f$ satisfies $\Delta\tilde f+(p-1)\tilde \rho \tilde f=0$, then it is equivalent to the following equations of $\phi_{k,i}$
\begin{align}
\begin{cases}\label{Akphik}
    A_k(\phi_{k,i}):= \phi_{k,i}''+\frac{N-1}{r}\phi_{k,i}'-\frac{\mu_k}{r^2}\phi_{k,i}+(p-1)\tilde \rho \phi_{k,i}=0\\
    \phi_{k,i}(1)=0, \phi'_{k,i}(0)=0
    \end{cases}
\end{align}
for $k=0,1,2,\cdots$, $i=1,\cdots, N_k$. We shall show that $\phi_{1,i}\in \text{span}\{\xi(r)\}$ and $\phi_{k,i}=0$ for any $k\neq 1$.

Step 1: Consider the case $k=0$.  Multiplying $A_0(\phi_0)=0$ by $\tilde U$ and applying integration by parts, one has $\int_{\Bn}\tilde U\phi_0\tilde \rho dx=0$. Since $\tilde U$ is positive in $\Bn$, then $\phi_0$ must changes sign. However, $A_0(\phi_0)=0$ implies $\Delta \phi_0<0$. By maximum principle and $\phi_0=0$ on the boundary, one has $\phi_0>0$ in $\Bn$. This is a contradiction unless $\phi_0=0$.

Step 2: Consider the case $k=1$. First, it is easy to see $\xi(r)$ satisfies \eqref{Akphik} for $k=1$.
Suppose there exists another solution $\phi(r)$ of $A_1(\phi)=0$. Since $\xi(r)<0$, then we can write $\phi(r)=c(r)\xi(r)$. By a straightforward calculation, one has 
\begin{align*}
    c''\xi+2c'\xi'+\frac{N-1}{r}c'\xi=0. 
\end{align*}
If $c(r)$ is  not constant, it follows that
\begin{align*}
    \frac{c''}{c'}=\frac{-2\xi'}{\xi}-\frac{N-1}{r}
\end{align*}
and hence $c'\sim 1/(r^{N-1}\xi^2)$ as $r\to 0$. Note that $\xi\sim r$ as $r\to 0$. Thus $c\sim r^{-N}$ as $r\to 0$. However, $c(r)\xi(r)$ does not belong to  $L^2(\Bn,\tilde \rho dx)$ for $N\geq 2$. Therefore $c$ must be a constant and $\phi\in \text{span}\{\xi\}$. 

Step 3: Consider the case $k\geq 2$. Note that $A_1$ has a solution $\xi(r)$ which does not change sign in $(0,1)$. By Sturm theorem, $A_1$ is a non-negative operator. Since $\mu_k>\mu_1$ if $k\geq 2$ and 
\begin{align*}
    A_k=A_1+\frac{\mu_k-\mu_1}{r^2},
\end{align*}
then $A_k$ is a positive operator for $k\geq 2$. Thus $A_k(\phi)=0$ implies $\phi=0$. Combining the above analysis, we have that if $\tilde f$ satisfies $\Delta \tilde f+(p-1)\tilde \rho \tilde f=0$, then 
$$\tilde f\in \text{span}\{\xi(r)Y_{1,1}(\theta),\cdots, \xi(r)Y_{1,N}(\theta)\}.$$ 
This completes the proof.
\end{proof}



\small
\bibliographystyle{plainnat}
\bibliography{ref.bib}

\end{document}